\documentclass[12pt,a4paper]{amsart}
\usepackage[margin=0.9in]{geometry}

\newcommand{\inv}[1]{\frac{1}{#1}}
\newcommand{\Prob}{\mathbb{P}}

\newcommand{\R}{\mathbb{R}}
\newcommand{\Rd}{\mathbb{R}^d}
\newcommand{\f}[2]{\frac{#1}{#2}}

\newcommand{\ps}[2]{\left\langle{#1},{#2}\right\rangle}

\renewcommand{\epsilon}{\varepsilon}

\newcounter{casenum}
\newcounter{subcasenum}
\addtocounter{subcasenum}{1}

\usepackage{amsthm}
\usepackage{mathtools}
\usepackage{amsmath,amssymb,amsfonts}
\usepackage{enumerate}

\theoremstyle{plain}
\newtheorem{thm}{Theorem}[section]

\newtheorem{prop}[thm]{Proposition}
\newtheorem{coro}[thm]{Corollary}
\newtheorem{lem}[thm]{Lemma}
\newtheorem{assump}[thm]{Assumptions}

\theoremstyle{definition}

\newtheorem{rmk}[thm]{Remark}
\title[Trend to equilibrium for granular equation]{Trend to equilibrium for granular media equations under non-convex potentials and application to log gases}
\author{Scander Mustapha}
\thanks{mustapha@princeton.edu}
\begin{document}
\maketitle
\begin{abstract}

We derive new HWI inequalities for the granular media equation, which external potential $V$ and interaction potential $W$ are only strictly convex on complementary parts of the space. Particularly, potentials are not assumed convex. After solving technicalities related to the singularity of a logarithmic $W$, we apply our result to obtain stability rates of log gases under non-strictly convex or quartic external potentials. We prove that the distribution of a log gas converges towards an equilibrium with respect to the Wasserstein distance at a square root rate. Finally, we establish exponential stability of log gases under the double-well potential $V(x) = \f{x^4}{4} + c\f{x^2}{2}$, $c < 0$ and the non-confining potential $V(x) = g\f{x^4}{4} + \f{x^2}{2}$, $g<0$ for $|c|$ and $|g|$ small enough. 
\end{abstract}

\section{Introduction}
\subsection{Granular media equation and HWI inequalities}
The present paper studies the extension of stability results proved for the granular media equation
\begin{equation}
    \label{eq:granular_media}
      \partial_t\mu = \nabla\cdot\left[
  \mu\nabla\left(\inv{2}V + W*\mu\right)
  \right]
\end{equation}
to non-convex potentials.
The unknown $\mu$ is a time-dependent probability measure on $\Rd$, $V:\Rd\rightarrow\R$ is an external potential and $W:\Rd\rightarrow\R$ is an interaction potential. The non-local and non-linear partial differential equation (\ref{eq:granular_media}) is the formal gradient flow of the entropy
\begin{equation}
\label{eq:entropy}
\Sigma(\mu) = \inv{2}\int_{\Rd}V(x)\mu(dx) + \inv{2}\int_{\Rd\times\Rd}W(x-y)\mu(dx)\mu(dy),
\end{equation}
which dissipation is defined as
\begin{equation}
\label{eq:information}
D(\mu) \equiv\int_{\Rd} \left|\nabla\left(\inv{2}V + W*\mu\right)\right|^2\mu(dx).
\end{equation}

Under the two sets of assumptions $D^2V\ge 2\lambda$ and $D^2W\ge 0$ or $D^2V\ge 0$ and $D^2W\ge\lambda$, for some $\lambda > 0$, Carrillo et al. established in~\cite{carrillo2003kinetic} the celebrated HWI inequality 
\begin{equation}
  \label{hwi}
\Sigma(\rho_0) - \Sigma(\rho_1) \le \sqrt{D(\rho_0)}W_2(\rho_0, \rho_1) - \f{\lambda}{2} W_2(\rho_0, \rho_1)^2,
\end{equation}
 $W_2$ being the Wasserstein distance and $\rho_0, \rho_1$ having finite entropy and belong to $\mathcal{M}_2$, the set of probability measures with finite second order moment. This inequality implies successively a transportation inequality, a log-Sobolev inequality and exponential stability with respect to the Wasserstein distance towards a minimizer $\mu_\infty$ of the entropy
\begin{equation}
  W_2(\mu_t, \mu_\infty)\le \sqrt{\f{2\left(\Sigma(\mu_0) - \Sigma(\mu_\infty)\right)}{\lambda}}e^{-\lambda t},\ \forall t\ge 0.
\end{equation}

The primary aim of this paper is to extend this method to non-strictly convex potentials. More precisely, we allow $V$ to be non-convex near the origin, according to the following assumptions.
\begin{assump}
  \label{assumpA}
  There exist $\alpha, \beta, \gamma, r > 0$ such that the $C^2$ potentials $V, W$ satisfy
  \begin{itemize}
  \item[(A1)] $D^2V(x)\ge - \beta$ for $x\in\Rd$,
  \item[(A2)] $D^2V(x) \ge \alpha$ for $|x|\ge r$,
  \item[(A3)] $V$ is symmetric: $V(x) = V(-x)$ for $x\in\Rd$,
  \item[(A4)] $W$ is convex and $D^2W(x) \ge \gamma$ for $|x|\le 2r$,
  \item[(A5)] $W$ is symmetric: $W(x) = W(-x)$ for $x\in\Rd$.
  \end{itemize}
\end{assump}
Under those assumptions, we prove the following

\begin{thm}
  \label{thm:hwi1}
  Assume that Assumptions \ref{assumpA} are satisfied.
  Let $\rho_0, \rho_1\in\mathcal{M}_2$ with finite entropy. Define
  \begin{equation}
    P_r = \max\left(\int_{|x|>r}\rho_0(dx), \int_{|x|> r}\rho_1(dx)\right).
  \end{equation}
  If $\rho_0$ and $\rho_1$ have the same center of mass, then
  \begin{equation}
    \label{eq:hwi}
    \Sigma(\rho_0) - \Sigma(\rho_1) \le \sqrt{D(\rho_0)} W_2(\rho_0, \rho_1)- \f{\lambda}{2} W_2(\rho_0,\rho_1)^2,
  \end{equation}
  where the constant $\lambda$ is given by
  \begin{equation}
      \label{lambda-fcm}
    \lambda = \f{\min\left(\alpha,2\gamma-\beta\right)}{2} - 2\gamma P_r.
  \end{equation}
  If $\rho_0$ and $\rho_1$ are symmetric, then (\ref{eq:hwi}) holds with the better constant $\lambda$:
    \begin{equation}
      \label{lambda-sym}
      \lambda = \f{\min\left(\alpha,2\gamma\big(1 - 2P_r)-\beta\right)}{2}.
    \end{equation}
\end{thm}
\begin{rmk}\
  \begin{enumerate}[(i)]
  \item The method of establishing HWI inequalities is not the only one to derive stability rates. We can cite the strategy of Bakry-Emery \cite{bakry}, which consists of computing the dissipation of the entropy dissipation or the method of characteristics or even ad-hoc computations similar to what is done in \cite{Ledoux_2009}. In any case, the computations turn out to be very similar.
  \item HWI inequality (\ref{eq:hwi}) is slightly different from usual HWI inequalities, because $\lambda$ depends on $\rho_0$ and $\rho_1$ through their tail probabilities. Therefore, application of our modified HWI inequality requires beforehand uniform bound of those tail probabilities, by the help of bound upon the moments for example. 
    \item Our proof relies heavily on exploiting the convexity of $W$, which requires the assumptions of a fixed center of mass or a symmetric initial data (see Theorem 2.2 \cite{carrillo2003kinetic} for example or \cite{bolley2012convergence} and \cite{bolley2013uniform}).
    \item Consider the internal energy $\mathcal{U}$
      \begin{equation}
        \mathcal{U}(\mu) = \int_{\Rd}U(\mu(x))\mu(dx),
      \end{equation}
      where $U:\R_+^*\to\R$ satisfies the dilation condition that $\lambda\in\R^*_+\to\lambda^dU(\lambda^{-d})$ is convex and non-increasing. The conclusions of Theorem \ref{thm:hwi1} hold with the same constants for the entropy
      \begin{equation}
        \mu\mapsto \mathcal{U}(\mu) + \Sigma(\mu).
      \end{equation}
      Typically for $U(\rho) = \rho\log{\rho}$, our approach can be used to prove exponential stability of Mckean-Vlasov diffusions under non-convex external potentials (see \cite{tugaut2013convergence}).
  \end{enumerate}
\end{rmk}

In the same fashion of Theorem 2.3 \cite{carrillo2003kinetic}, we investigate the case of $V$ convex (non necessarily strictly convex) and $W$ degeneratly convex at infinity.
\begin{assump}
  \label{assumpB}
  $V$ and $W$ belong to $C^2(\Rd)$ and $C^2(\Rd-\{0\})$ respectively and there exist positive constants $c$ and $\eta$ such that
  \begin{itemize}
  \item[(B1)] $D^2V(x)\ge 0$ for $x\in\R^d$,
  \item[(B2)] $D^2W(x) \ge \f{c}{|x|^\eta}$ for $x\in\R^d-\{0\}$.
  \end{itemize}
\end{assump}
Under those assumptions and additional technical Assumptions~\ref{assumpD}, postponed to Section \ref{section-4}, we prove the following 
\begin{thm}\label{thm:hwi2}
  Assume that Assumptions~\ref{assumpB} and~\ref{assumpD} are satisfied. Let $\rho_0, \rho_1\in\mathcal{M}_2$ with finite fourth moments and finite entropy. Define
  \begin{equation}
    m = \max{\left(\int_{\Rd} |x|^4\rho_0(dx),\int_{\Rd} |x|^4\rho_1(dx)\right)}.
  \end{equation}
 Then, there exists a constant $C > 0$ depending only on $m$, $c$ and $\eta$ such that the following HWI inequality holds
  \begin{equation}
    \label{eq:hwi2}
    \Sigma(\rho_0) - \Sigma(\rho_1) \le \sqrt{D(\rho_0)}W_2(\rho_0, \rho_1) - CW_2(\rho_0, \rho_1)^{\eta + 2}.
  \end{equation}
\end{thm}
This variant of the HWI inequality (\ref{eq:hwi}) implies algebraic stability, meaning
\begin{equation}\label{result-sqrt}
  W_2(\mu_t, \mu_\infty) \le \f{C}{t^{1/\eta}},\ \forall t\ge 1.
\end{equation}

\subsection{Application to log gases}
The second contribution of this paper is the application of the previous results to log gases. Those gases are obtained by taking $W = -\log{|\cdot|}$ in the uni-dimensional case $d = 1$. It leads to the Fokker-Planck equation
\begin{equation}
\label{eq:fokker_planck}
\partial_t\mu_t = \partial_x\left[\mu_t\left(\inv{2}V' - H\mu_t\right)\right],
\end{equation}
where $H$ denotes the Hilbert transform.
The logarithmic potential has a singularity at the origin and is only convex on the half-lines $\R^*_+$ and $\R^*_-$. Because of these difficulties, the previous theorems do not apply directly and the derivation of the analogue of (\ref{eq:hwi}) is more involved. This is the purpose of Theorem \ref{thm:ramon}, which extends the application of HWI inequalities to log gases. 
\begin{thm}\label{thm:ramon}
  Let $V$ be a $C^2$ symmetric potential satisfying Assumptions (A1)- (A3) and let $W = -\log{|\cdot|}$.
  Then the conclusions of Theorem~\ref{thm:hwi1} hold for probability measures in $L^{\infty}(\R)$, with $\gamma =\inv{4r^{2}}$.
\end{thm}
This result and its proof have two main applications. The first concerns log gases under a general convex potential $V$. Under appropiate growth assumption of $V$, ensuring the existence and uniqueness of the minimizer of the entropy $\mu_{V}$, we prove that any solution converges towards $\mu_V$ at a square root rate. The definition of a solution to (\ref{eq:fokker_planck}) is given afterwards in Section~\ref{section-5}.
\begin{thm}\label{thm:log-alg}
  Let $V$ be a $C^2$ convex potential satisfying growth Assumptions~\ref{assump:growth}. Let $\mu_V$ the unique minimizer of the entropy and $(\mu_t)_{t\ge 0}$ be a solution of (\ref{eq:fokker_planck}). Then we have algebraic stability towards $\mu_V$
  \begin{equation}
    W_2(\mu_t, \mu_V)\le \f{C}{\sqrt{t}},\ \forall t\ge 1,
  \end{equation}
  for some constant $C>0$.
\end{thm}
To our knowledge, in the case of log gases, the weakest assumptions required for an explicit equilibrium were strict convexity at least away of the origin (see~\cite{Ledoux_2009}). Our result weakens this assumption, the cost being a slower rate.

\subsection{Log gases with quartic potentials}
The second application of Theorem \ref{thm:ramon} concerns quartic potentials $V$, defined by
\begin{assump} For some constants $g, c\in(-\infty, 0)$, $(C1)$ or $(C2)$ is satisfied
  \begin{enumerate}
    \item[(C1)] $V(x) = \f{x^4}{4} + c\f{x^2}{2}$ for $x\in\R$,
    \item[(C2)] $V(x) = g\f{x^4}{4} + \f{x^2}{2}$ for $x\in\R$.
  \end{enumerate}
\end{assump}
For $c < 0$, quartic potential (C1) is non-convex and does not fall under the scope of the work of Ledoux and Popescu ~\cite{Ledoux_2009}. After deriving moment estimates (which already imply stability and provide a simple proof to Theorem 1.1~\cite{donati2018convergence}), we apply Theorem~\ref{thm:ramon} to derive an exponential stability rate towards the unique minimizer $\mu_{V}$ of the entropy, for solutions with a fixed center of mass or with a symmetric initial data. 
\begin{thm}\label{thm:confining}
  Let $c\in\left(-\f{1}{4\sqrt{17}}, 0\right)$ and $V$ be defined by $(C1)$. Let $(\mu_t)_{t\ge 0}$ be a solution of (\ref{eq:fokker_planck}) with a fixed center of mass and finite fourth moments. Assume that the moment condition 
    \begin{equation}
      \label{moment-cond}
      \int_{\R}x^2\mu_0(dx) \le \f{-c + \sqrt{c^2+4}}{2},
    \end{equation}
    is satisfied. Then $(\mu_t)_{t\ge 0}$ is exponentially stable towards $\mu_V$ with respect to the Wasserstein distance
    \begin{equation}
      W_2(\mu_t, \mu_V)\le \sqrt{\f{2(\Sigma(\mu_0)- \Sigma(\mu_V))}{\lambda}}e^{-\lambda t},\ \forall t\ge 0.
    \end{equation}
    The rate $\lambda$ is given by
    \begin{equation}
          \lambda = \inv{16\left(-c+\sqrt{c^2+4}\right)} + \f{c}{2} > 0.
    \end{equation}
    Similarly, if $c\in \left(-\inv{\sqrt{6}}, 0\right)$ and if $(\mu_t)_{t\ge 0}$ is a solution of (\ref{eq:fokker_planck}) with finite fourth moments and a symmetric initial data $\mu_0$, then under the same moment condition, $(\mu_t)_{t\ge 0}$ is exponentially stable towards $\mu_V$ with respect to the Wasserstein distance with rate $\lambda$ given by
    \begin{equation}
          \lambda = \inv{2\left(-c+\sqrt{c^2+4}\right)} + \f{c}{2} > 0.
    \end{equation}
\end{thm}

The assumption of a fixed center of mass or a symmetric initial data are required to exploit easily the strict convexity of $W$. It leads to tractable computations and exact numerical values. We are able to relax those assumptions in the following theorem, and assume only lower-bounded second moments for the initial data. We prove exponential stability under weaker asssumptions but only for $c \in (c^*, 0)$ with $c^{*}\sim10^{-9}$.
\begin{thm}\label{thm:bonus}
  There exists $c^* < 0$, such that if $c\in(-c^{*}, 0)$ and if $(\mu_t)_{t\ge 0}$ is a solution of (\ref{eq:fokker_planck}) with finite sixth moments and initial data satisfying 
  \begin{equation}
    \left(\f{2}{-c+\sqrt{c^2+16}}\right)^4 \le \int x^2\mu_0(dx)
  \end{equation}
  and 
  \begin{equation}
    \int x^4\mu_0(dx)\le  \left(\f{-c+\sqrt{c^2+12}}{2}\right)^2,
  \end{equation}
  then $(\mu_t)_{t\ge 0}$ is exponentially stable towards the equilibrium $\mu_V$.
\end{thm}

For $g < 0$, a new difficulty arises due to the fact that the quartic potential (C2) is neither convex nor confining. Nevertheless, we manage to establish exponential stability for $g\in \left(-\inv{81+36\sqrt{5}}, 0\right)$ and for well-defined solutions. This result is stated in the following theorem and answers a conjecture formulated in \cite{biane2001free} (see Conjecture 7.3 \cite{li2013generalized} as well).

\begin{thm}
  \label{thm:biane}
  Let $g\in \left(- \inv{81+ 36\sqrt{5}}, 0\right)$ and $m > 0$ satisfying
  \begin{equation}
    m <\sqrt{- \inv{3g} - \f{4}{\sqrt{-g}} - 3}.
  \end{equation}
  Let $\mu_0$ be an initial measure with $supp(\mu_0)\subset [-m, m]$. Then any solution $(\mu_t)_{t\ge 0}$ of (\ref{eq:fokker_planck}) with quartic $V$ (C2) is well-defined and converges exponentially towards a stationary measure $\mu_\infty$
  \begin{equation}
  W_2(\mu_t, \mu_\infty) \le \sqrt{\f{2\left(\Sigma(\mu_0) - \Sigma(\mu_\infty)\right)}{\lambda}}e^{-\lambda t},\ \forall t\ge 0,
  \end{equation}
  with rate $\lambda$ given by
  \begin{equation}
    \lambda = \inv{2}\left[1 + 3g\left(m^2 + \f{4}{\sqrt{-g}} + 3\right)\right] > 0.
  \end{equation}
  Moreover, $\mu_\infty$ is a local minimizer of the entropy:
  \begin{equation}
    \forall \mu\in \mathcal{M}_2, \ supp(\mu)\subset \left( -\sqrt{m^2 + \f{4}{\sqrt{-g}} + 3},  \sqrt{m^2 + \f{4}{\sqrt{-g}} + 3}\right)\implies \Sigma(\mu_\infty)\le \Sigma(\mu).
  \end{equation}
\end{thm}

The paper is structured as follows. In Section \ref{section-2}, we recall some preliminary facts collected from \cite{villani2003topics} and \cite{carrillo2003kinetic}, and we introduce notation. Section \ref{section-3} contains the proof of Theorem \ref{thm:hwi1} and its corollary concerning stability of solutions. In Section \ref{section-4}, we establish the proof of Theorem \ref{thm:hwi2} and its implications. In Section \ref{section-5}, we consider log gases and prove Theorems \ref{thm:ramon} - \ref{thm:biane}. The appendix gathers auxiliary proofs.

 \section{Preliminaries and notations}\label{section-2}
 In the whole paper, we denote by $\mathcal{M}$ the set of probability measures on $\Rd$, and we define for $p\ge 1$
 \begin{equation}
   \mathcal{M}_p = \left\{\mu\in\mathcal{M} : \int_{x\in\Rd}|x|^p\mu(dx) < \infty\right\}.
 \end{equation}
We recall the definition of the Wasserstein metric $W_2$ on $\mathcal{M}_2$
\[W_2(\rho_0, \rho_1) = \bigg[\inf_{\gamma\in\Gamma(\rho_0, \rho_1)}\int_{\Rd\times\Rd}|x-y|^2\gamma(dx, dy)\bigg]^{1/2},\]
where $\Gamma(\rho_0, \rho_1)$ denotes the set of couplings between $\rho_0$ and $\rho_1$ (see \cite{villani2003topics}). According to Brenier Theorem, if $\rho_0$ and $\rho_1$ have a density, there exists a unique optimal transport map $T = \nabla \phi$, gradient of a convex function $\phi:\Rd\rightarrow\R$, $\rho_0$-almost everywhere, such that $\rho_1$ is the push-forward measure $\rho_1 = T\# \rho_0$ and such that
\begin{equation}
\label{eq:optimal_T}
    W_2(\rho_0, \rho_1) = \bigg[\int_{\Rd}|x - T(x)|^2\rho_0(dx)\bigg]^{1/2}.
\end{equation}

Recall from \cite{mccann1997convexity}, that a functional $F:\mathcal{M}_2\rightarrow\R$ is said to be displacement convex if $s\mapsto F(\rho_s)$ is a convex function, where $(\rho_s)_{s\in [0, 1]}$ is the geodesic in $(\mathcal{M}_2, W_2)$ joining $\rho_0$ to $\rho_1$:
\begin{equation}
    \label{eq:interpolation}
    (\rho_s)_{s\in [0, 1]} = ([(1-s)Id + sT]\#\rho_0)_{s\in [0, 1]}.
\end{equation}

Strict displacement convexity is a stronger property, which plays a crucial role in establishing equilibrium rates. The key to derive a HWI inequality is to prove that $\Sigma$ is $\lambda$-strictly convex along interpolation (\ref{eq:interpolation}):
\[\frac{d^2}{ds^2}\Sigma(\rho_s) \ge \lambda W_2(\rho_0, \rho_1)^2,\ 0 < s < 1.\]

We denote $\mathcal{V}$ and $\mathcal{W}$ the functionals
\begin{equation}
  \begin{split}
  \mathcal{V} &: \mu\in\mathcal{M} \to \inv{2}\int V(x)\mu(dx),\\
    \mathcal{W} &: \mu\in\mathcal{M}\to \inv{2}\iint W(x-y)\mu(dx)\mu(dy).
  \end{split}
\end{equation}

We define $(\mu_t)_{t\ge 0}\in C\left(\R_+, \mathcal{M}\right)$ as a solution of (\ref{eq:granular_media}) with initial data $\mu_0$ if for all $t\ge 0$, $\mu_t$ has a density, that we shall denote $\mu_t(x)$, such that $\nabla W *\mu_t\in L^\infty_{loc}(\R_+\times\Rd)$ and
\begin{equation}
  \int \phi d\mu_t - \int\phi d\mu_0 = -\int_0^tds\int\nabla\phi\cdot\nabla\left(\inv{2}V + W*\mu_s\right)d\mu_s,\ \forall \phi\in C_0^\infty(\Rd),
\end{equation}
where $C_0^\infty(\Rd)$ denotes the space of smooth and compactly supported test functions. 

Finally, we end this section by recalling that Proposition 2.1 \cite{carrillo2003kinetic} ensures existence of solutions to (\ref{eq:granular_media}), under additional technical assumptions concerning the regularity and the growth at infinity of the $C^2$ potentials $V$ and $W$, and establishes the dissipation property
\begin{equation}
  \label{eq:dissip}
  \f{d}{dt}\Sigma(\mu_t)\le - D(\mu_t),\ \forall t\ge 0,
\end{equation}
for $(\mu_t)_{t\ge 0}$ a solution. Moreover, it is straigthforward to prove that $\Sigma$ and $D$ are lower-semi continuous for the weak topology. In Section \ref{section-3}, we assume that those technical assumptions are satisfied. However, potentials $W$, as in Theorem \ref{thm:hwi2}, are not $C^2$ and Proposition 2.1 \cite{carrillo2003kinetic} cannot be applied.
As the purpose of this work is not to prove the dissipation property nor to discuss the existence of solutions, additional assumptions will be therefore assumed to guarantee the validity of Proposition 2.1~\cite{carrillo2003kinetic}.

\section{Proof of Theorem \ref{thm:hwi1}}\label{section-3}
\begin{proof}[Proof of Theorem \ref{thm:hwi1}]
  The main difficulty is to alleviate non-convexity of $V$ near the origin with the strict convexity of $W$, and conversely to use the strict convexity of $V$ outside a neighborhood of the origin to alleviate non-strict convexity of $W$ outside the origin.
  
  Let $\rho_0,\rho_1\in \mathcal{M}_2$ with finite entropy and the same center of mass. Let $T$ be the optimal transport map from $\rho_0$ to $\rho_1$ described by (\ref{eq:optimal_T}). Consider interpolation (\ref{eq:interpolation}) between $\rho_0$ and $\rho_1$ given by $(\rho_s)_{s\in [0, 1]}$.
  Set $\theta(x) = T(x)-x$. In these circumstances, we have
  \[ W_2(\rho_0, \rho_1)^2 = \int_{\R^d}|\theta(x)|^2\rho_0(dx) \]
  and
  \[ \int_{\R^d}f(x)\rho_s(dx) = \int_{\R^d}f(x + s\theta(x))\rho_0(dx)\]
  for all measurable bounded functions $f$.

  Taylor's formula applied to $\Sigma(\rho_t)$ between 0 and 1 gives
  \begin{equation}
    \label{eq:taylor}
    \Sigma(\rho_1) - \Sigma(\rho_0) = \frac{d}{ds}\bigg\vert_{0}\Sigma(\rho_s) + \frac{1}{2}\frac{d^2}{ds^2}\bigg\vert_{s^*}\Sigma(\rho_s),
  \end{equation}
  for some $s^*\in (0, 1)$.
  Following computations of Section 4.1 \cite{carrillo2003kinetic}, we find for all $s\in (0, 1)$
  \begin{equation}
    \label{eq:taylor_1}
    \frac{d}{ds}\bigg\vert_{0}\Sigma(\rho_s)\ge -\sqrt{D(\rho_0)}W_2(\rho_1,\rho_0)
  \end{equation}
  and
  \begin{equation}
    \label{eq:taylor_2}
    \begin{split}
      \frac{d^2}{ds^2}\Sigma(\rho_s)&\ge \underbrace{\inv{2}\int_{\Rd}\ps{D^2V(x + s\theta(x))\cdot\theta(x)}{\theta(x)}\rho_0(dx)}_{(\ref{eq:taylor_2}.1)}\\
      &+\underbrace{\inv{2}\int_{\R^{2d}}\ps{D^2W(x-y + s(\theta(x) - \theta(y))\cdot(\theta(x) - \theta(y))}{\theta(x) - \theta(y)}\rho_0(dx)\rho_0(dy)}_{(\ref{eq:taylor_2}.2)}.\\
    \end{split}
  \end{equation}

  On the one hand, we have using Assumptions $(A1)$ and $(A2)$
  \begin{equation}
    \label{eq:taylor_2_1}
   (\ref{eq:taylor_2}.1) \ge \f{\alpha}{2}\int_{|x + s\theta(x)|\ge r}|\theta(x)|^2\rho_0(dx) - \f{\beta}{2}\int_{|x+s\theta(x)|<r}|\theta(x)|^2\rho_0(dx).
  \end{equation}

  On the other hand, under $(A3)$
  \begin{equation}
    \label{eq:taylor_2_2}
    \begin{split}
      (\ref{eq:taylor_2}.2) &\ge\inv{2} \int\limits_{\substack{|x+s\theta(x)|\le r\\|y+s\theta(y)|\le r}}\ps{D^2W(x - y + s(\theta(x)-\theta(y)))\cdot(\theta(x)-\theta(y))}{\theta(x)-\theta(y)}\rho_0(dx)\rho_0(dx)\\
      &\ge\inv{2}\gamma\int\limits_{\substack{|x+s\theta(x)|\le r\\|y+s\theta(y)|\le r}}\big(|\theta(x)|^2 + |\theta(y)|^2 - 2\ps{\theta(x)}{\theta(y)}\big)\rho_0(dx)\rho_0(dy)\\
      &\ge\gamma\int\limits_{\substack{|x+s\theta(x)|\le r}}\rho_0(dx) \int\limits_{\substack{|x+s\theta(x)|\le r}}|\theta(x)|^2\rho_0(dx) - \gamma\left|\ \int\limits_{\substack{|x+s\theta(x)|\le r}}\theta(x)\rho_0(dx)\right|^2.\\
    \end{split}
  \end{equation}

Using the fact that $\int_{\Rd}x\rho_0(dx) = \int_{\Rd}x\rho_1(dx) $, we can write
\begin{equation}
    \label{eq:center_of_mass}
    \int\limits_{\substack{|x+s\theta(x)|\le r}}\theta(x)\rho_0(dx)=-\int\limits_{\substack{|x+s\theta(x)|> r}}\theta(x)\rho_0(dx)
\end{equation}
and by Cauchy-Schwarz inequality
\begin{equation}
    \label{eq:double_term}
    \gamma\left|\ \int\limits_{\substack{|x+s\theta(x)|\le r}}\theta(x)\rho_0(dx)\right|^2 = \gamma\left|\ \int\limits_{\substack{|x+s\theta(x)|> r}}\theta(x)\rho_0(dx)\right|^2 \le \gamma\int\limits_{\substack{|x+s\theta(x)|> r}}\rho_0(dx)\int\limits_{\substack{|x+s\theta(x)|>r}}|\theta(x)|^2\rho_0(dx).
\end{equation}

The first integral in the left-hand side is equal to $\int_{|x|> r}\rho_s(dx)$, and therefore combining estimations (\ref{eq:taylor_2_2}) and (\ref{eq:double_term}), it follows
\begin{equation}
     \label{eq:taylor_2_2_bis}
(\ref{eq:taylor_2}.2) \ge \gamma\left[\left(1-\int\limits_{|x|> r}\rho_s(dx)\right)\int\limits_{\substack{|x+s\theta(x)|\le r}}|\theta(x)|^2\rho_0(dx)  - \int\limits_{|x|> r}\rho_s(dx)\int\limits_{\substack{|x+s\theta(x)|> r}}|\theta(x)|^2\rho_0(dx) \right]
\end{equation}

Noticing that for $s\in [0, 1]$
\begin{equation}
  \label{estimate-rho-s}
 \int_{|x|> r}\rho_s(dx)\le \int_{|x|>r}\rho_0(dx) + \int_{|T(x)|>r}\rho_0(dx) = \int_{|x|> r}\rho_0(dx)+ \int_{|x|> r}\rho_1(dx) \le 2P_r 
\end{equation}
we conclude from (\ref{eq:taylor_2}), (\ref{eq:taylor_2_1}), (\ref{eq:taylor_2_2_bis}) and (\ref{estimate-rho-s}), by setting
  \[\lambda =\f{\min\left(\alpha, 2\gamma -\beta\right)}{2} -2\gamma P_r\]
  that
  \begin{equation}
    \label{eq:taylor_2_final}
    \frac{d^2}{ds^2}\bigg\vert_{s^*}\Sigma(\rho_s) \ge \lambda\int_{\Rd}|\theta(x)|^2\rho_0(dx) = \lambda W_2(\rho_1,\rho_0)^2.
  \end{equation}

  Combining (\ref{eq:taylor}), (\ref{eq:taylor_1}) (\ref{eq:taylor_2_final}), yields
  \begin{equation}
  \label{eq:hwi_proof}
        \Sigma(\rho_0) - \Sigma(\rho_1) \le \sqrt{D(\rho_0)}W_2(\rho_0,\rho_1)-\f{\lambda }{2}W_2(\rho_0,\rho_1)^2.
  \end{equation}
  
  In the special case when $\rho_0$ and $\rho_1$ are symmetric, we observe that the quantity (\ref{eq:center_of_mass}) vanishes. Indeed, define
  \[ A = \{x\in\Rd: |(1-s)x+sT(x)|\le r\}\]
  and
  \[ c = \int_{A}(x - T(x))\rho_0(dx) =  \int\limits_{\substack{|x+s\theta(x)|\le r}}\theta(x)\rho_0(dx).\]
  From the uniqueness of the optimal transport map $T$, the symmetry of $\rho_0$ and $\rho_1$ and the fact that $W_2(\rho_0, \rho_1)^2 = \int_{\Rd}(x + T(-x))^2\rho_0(dx)$, we see that the map $T$ is odd. Consequently, $A$ is a symmetric domain and $x\mapsto x - T(x)$ is odd. Therefore, $c = 0$ and (\ref{eq:hwi_proof}) holds with
     \[\lambda \equiv \f{\min\left(\alpha,2\gamma\big(1 - 2P_r)-\beta\right)}{2}.\]     
\end{proof}

We derive the following  asymptotic behavior and inequalities from (\ref{eq:hwi}).

\begin{coro}
\label{coro:hwi1}
Assume that Assumptions \ref{assumpA} are satisfied and that $\Sigma$ is lower-bounded. Let $(\mu_t)_{t\ge 0}$ be a solution of (\ref{eq:granular_media}) with a fixed center of mass. Define
\begin{equation}
  P_r = \sup_{t\ge 0}\int_{|x|>r}\mu_t(dx)
\end{equation}
and assume that 
  \begin{equation}
    \lambda = \f{\min\left(\alpha,2\gamma-\beta\right)}{2} - 2\gamma P_r > 0.
  \end{equation}
If the family $(\mu_t)_{t\ge 0}$  is tight with respect to the weak topology then $(\mu_t)_{t\ge 0}$ exponentially converges, with respect to the Wasserstein distance, to the unique minimizer $\mu_\infty$ of the entropy $\Sigma$ among the class of probability measures $\rho$ satisfying
  \begin{equation}
  \label{eq:centertailcond}
      \int_{\Rd}x\rho(dx) = \int_{\Rd}x\mu_0(dx)\ \textrm{ and }\ \Prob_{\rho}(|x|> r)\le \sup_{t\ge 0}\Prob_{\mu_t}(|x|> r).
  \end{equation}
  Moreover, the following inequalities hold:
  \begin{enumerate}[(i)]
  \item Logarithmic Sobolev inequality
    \begin{equation}
      \label{eq:logsob}
      2\lambda\left(\Sigma(\mu_t) -\Sigma(\mu_\infty)\right)\le D(\mu_t),\ \forall t\ge 0.
    \end{equation}
  \item Transportation inequality
    \begin{equation}
      \label{eq:transp}
    W_2(\mu_t, \mu_\infty)\le\sqrt{\frac{2\left(\Sigma(\mu_t) - \Sigma(\mu_\infty)\right)}{\lambda}},\ \forall t\ge 0.
    \end{equation}
    \item Exponential stability
        \begin{equation}
        \label{eq:rateofcv}
    W_2(\mu_t, \mu_\infty) \le \sqrt{\frac{2\left(\Sigma(\mu_t) - \Sigma(\mu_\infty)\right)}{\lambda}}e^{-\lambda t},\ \forall t\ge 0.
        \end{equation}
  \end{enumerate}
  Finally, if $\mu_0$ is symmetric and
  \begin{equation}
      \lambda = \f{\min\left(\alpha,2\gamma\big(1 - 2P_r)-\beta\right)}{2} > 0,
  \end{equation}
  then the same inequalities hold with this better rate.
\end{coro}
\begin{proof}
Let $\mu_\infty$ be a limit point of $(\mu_t)_{t\ge 0}$ for the weak topology. By lower-semi-continuity of $D$
  \[ \limsup_{t\rightarrow\infty}\frac{d}{dt}\Sigma(\mu_t) \le -\liminf_{t\rightarrow\infty}D(\mu_t) \le - D(\mu_\infty).\]
If $-D(\mu_\infty)< 0$, then $\limsup_{t\rightarrow\infty}\frac{d}{dt}\Sigma(\mu_t) \le- D(\mu_\infty) < 0$ and there exist $t_0$ and $c$ such that $\Sigma(\mu_t) \le -\frac{1}{2}D(\mu_\infty)t + c$ for $t > t_0$. That is $\Sigma(\mu_t) \xrightarrow{t\rightarrow\infty} -\infty$. As $\Sigma$ is bounded below by assumption, $D(\mu_\infty) = 0$ and $\mu_\infty$ is a stationary solution.

Moreover, by weak convergence, $\mu_\infty$  will satisfy conditions (\ref{eq:centertailcond}). We can therefore apply Theorem \ref{thm:hwi1} to $(\rho_0, \rho_1)=(\mu_t, \mu_\infty)$ (notice that in the case of $\mu_0$ symmetric, $\mu_t$ stays symmetric at all times for $t>0$ by symmetry of the potentials). According to the HWI inequality (\ref{eq:hwi})
\[ \Sigma(\mu_t) - \Sigma(\mu_\infty) - \sqrt{D(\mu_t)}W_2(\mu_t, \mu_\infty) +  \f{\lambda}{2} W_2(\mu_t, \mu_\infty)^2\le 0,\ \forall t\ge 0.\]
Consequently the following discriminant is non-negative
\begin{equation}
D(\mu_t) - 2\lambda\left(\Sigma(\mu_t) - \Sigma(\mu_\infty)\right) \ge 0,\ \forall t\ge 0
\end{equation}
and the log-Sobolev inequality (\ref{eq:logsob}) holds.

For the transportation inequality (\ref{eq:transp}), take $\rho_1 = \mu_t$ and $\rho_0 = \mu_\infty$. $\mu_\infty$ being a stationary solution of (\ref{eq:granular_media}), $D(\rho_0) = 0$, which gives
\begin{equation}
    \label{eq:proof_mini}
     W_2(\mu_t,\mu_\infty)^2 \le \f{2\left(\Sigma(\mu_t) - \Sigma(\mu_\infty)\right)}{\lambda},\ \forall t\ge 0.
\end{equation}
This proves the transportation inequality. Since measure $\mu_t$ can be replaced by any measure $\rho$ satisfying (\ref{eq:centertailcond}) (apply Theorem \ref{thm:hwi1} to $(\mu_\infty, \rho)$), we have
\begin{equation}
\Sigma(\rho) - \Sigma(\mu_\infty)  \ge 0
\end{equation}
and if $\Sigma(\rho) = \Sigma(\mu_\infty)$ it follows that $W_2(\rho,\mu_\infty)=0$.
Therefore,  $\mu_\infty$  is the unique minimizer of the entropy $\Sigma$ among the class of probability measures $\rho$ satisfying (\ref{eq:centertailcond}).

Finally, to prove (\ref{eq:rateofcv}), use successively the log-Sobolev inequality (\ref{eq:logsob}), property (\ref{eq:dissip}), Gronwall's lemma and the transportation inequality (\ref{eq:transp}).
\end{proof}

\section{Proof of Theorem~\ref{thm:hwi2}}\label{section-4}
As explained at the end of Section~\ref{section-2}, the singularity of $W$ at the origin requires additional results, that should be proved for each $W$ of application. Indeed, potentials of interest which satisfy $D^2W(x)\ge \f{c}{|x|^\eta}$ like $W=-\log{|\cdot|}$ or $W = |\cdot|^p$, $p \in[0, 2)$, do not satisfy the assumptions of Proposition 2.1~\cite{carrillo2003kinetic}. Therefore, the dissipation property (\ref{eq:dissip}) does not hold necessarily. Moreover, formula (\ref{eq:taylor_2}) is not justified. To overcome those technical difficulties, we introduce the following additional assumptions.
\begin{assump} $V\in C^2(\Rd)$ and $W\in C^2(\Rd-\{0\})$ satisfy
  \label{assumpD}
  \begin{itemize}
  \item[(D1)] Dissipation property (\ref{eq:dissip}):
    \begin{equation}
      \f{d}{dt}\Sigma(\mu_t)\le - D(\mu_t),\ \forall t\ge 0.
    \end{equation}
  \item[(D2)] For any geodesic $(\rho_s)_{0\le s\le 1} = (\left[(1-s)Id + sT\right]\#\rho_0)_{0\le s\le 1}$, $\rho_0, \rho_1\in\mathcal{M}_{2}$ with finite entropy, the function $s\in[0, 1]\to \mathcal{W}({\rho_s})$ is twice differentiable and for all $s\in (0, 1)$
    \begin{equation}
      \frac{d^2}{ds^2}\mathcal{W}(\rho_s)\ge\inv{2}\int\limits_{\substack{x,y\in\R^{d}\\x\neq y}}\ps{D^2W(x-y + s(\theta(x) - \theta(y))\cdot(\theta(x) - \theta(y))}{\theta(x) - \theta(y)}\rho_0(dx)\rho_0(dy)
    \end{equation}
    where $\theta(x) = x - T(x)$.
  \end{itemize}
  
\end{assump}

We are now ready for the proof of Theorem \ref{thm:hwi2}.
\begin{proof}[Proof of Theorem \ref{thm:hwi2}]
  Under Assumptions \ref{assumpB}, $\mathcal{V}$ is convex-displacement
  \begin{equation}
    \frac{d^2}{ds^2}\mathcal{V}(\rho_s)\ge 0.
  \end{equation}

  In order to treat $\mathcal{W}$, we follow the proof of Theorem \ref{thm:hwi1} by fixing some $r>0$ and taking $\gamma = \inv{(2r)^\eta}$. We have immediately
\begin{equation}
 \frac{d^2}{ds^2}\mathcal{W}(\rho_s)\ge\f{c}{(2r)^\eta}\left[\left(1-\mathbb{P}_{\rho_s}(|x|> r)\right)\int\limits_{\substack{|x+s\theta(x)|\le r}}|\theta(x)|^2\rho_0(dx)  - \mathbb{P}_{\rho_s}(|x|> r)\int\limits_{\substack{|x+s\theta(x)|\ge r}}|\theta(x)|^2\rho_0(dx) \right]
\end{equation}
and therefore
\begin{equation}
  \frac{d^2}{ds^2}\mathcal{W}(\rho_s)\ge \f{c}{(2r)^\eta}\mathbb{P}_{\rho_s}(|x|\le r)\int_{\Rd}|\theta(x)|^2\rho_0(dx)  - \f{c}{(2r)^\eta}\int\limits_{\substack{|x+s\theta(x)|\ge r}}|\theta(x)|^2\rho_0(dx).
\end{equation}

By Cauchy-Schwarz inequality, we have the estimate
\begin{equation}
\int\limits_{\substack{|x+s\theta(x)|\ge r}}|\theta(x)|^2\rho_0(dx)\le \sqrt{\int_{\Rd}|\theta(x)|^4\rho_0(dx)}\sqrt{\Prob_{\rho_s}(|x|\ge r)}.
\end{equation}
Using the fact that
\begin{equation}
  \int_{\Rd}|\theta(x)|^4\rho_0(dx)\le 8m,
\end{equation}
we deduce
\begin{equation}
 \frac{d^2}{ds^2}\mathcal{W}(\rho_s)\ge \frac{c}{(2r)^\eta}\mathbb{P}_{\rho_s}(|x|\le r)W^2_2(\rho_0, \rho_1)  - \f{8^{1/4}m^{1/4}c}{(2r)^\eta}\Prob_{\rho_s}(|x|\ge r)^{1/4}W_2(\rho_0, \rho_1).
\end{equation}

The tail probability $\Prob_{\rho_s}(|x|\ge r)$ can be estimated by
\begin{equation}
  \Prob_{\rho_s}(|x|\ge r)\le \inv{r^4}\int_{\Rd}|sx + (1-s)T(x)|^4\rho_0(dx)\le \frac{8m}{r^4}.
\end{equation}
Moreover, with the simple estimate $W_2(\rho_0, \rho_1)^2\le 2\sqrt{m}$, we get
\begin{equation}
 \frac{d^2}{ds^2}\mathcal{W}(\rho_s) \ge\f{c}{(2r)^\eta}W^2_2(\rho_0, \rho_1)  - \frac{\sqrt{8m}c}{2^\eta r^{\eta+1}}W_2(\rho_0, \rho_1) - \frac{4mc}{2^\eta r^{\eta+2}}.
\end{equation}

Choosing optimally $r$ yields for some constant $C>0$ depending only on $\eta$, $m$ and $c$ that
\begin{equation}
 \frac{d^2}{ds^2}\mathcal{W}(\rho_s) \ge C W_2(\rho_0, \rho_1)^{\eta + 2},\ \forall s\in (0, 1),
\end{equation}
from which we deduce the desired HWI inequality.
\end{proof}

\begin{coro}
  \label{coro:hwi2}
  Assume that Assumptions \ref{assumpB} and \ref{assumpD} are satisfied. Let $(\mu_t)_{t\ge 0}$ be a solution of (\ref{eq:granular_media})  with uniformly fourth order moments. Set
  \begin{equation}
    m = \sup_{t\ge 0}\max{\left(\int_{\Rd}|x|^4\mu_t(dx)\right)}.
  \end{equation}
  Then, the following holds for some positive constants $\delta, C$ depending only on $m$, $c$ and $\eta$:
  \begin{enumerate}[(i)]
  \item Algebraic decay of the entropy
    \begin{equation}\label{decay-estimate}
  \Sigma(\mu_t)-\Sigma(\mu_\infty)\le \frac{\Sigma(\mu_0)-\Sigma(\mu_\infty)}{\left[1 + \delta \left(\Sigma(\mu_0)-\Sigma(\mu_\infty)\right)^{\eta/(\eta + 2)} t\right]^{(\eta + 2)/\eta}},\ \forall t\ge 0.
    \end{equation}
  \item Transportation inequality
    \begin{equation}\label{transpo}
  W_2(\mu_t, \mu_\infty) \le C\left(\Sigma(\mu_t)-\Sigma(\mu_\infty)\right)^{\inv{\eta+2}},\ \forall t\ge 0.
    \end{equation}
  \end{enumerate}

  Therefore, we have algebraic stability with respect to $W_2$ towards the minimizer $\mu_\infty$ of the entropy among the class of probability measures $\rho$ satisfying
  \begin{equation}
 \sup_{t\ge 0}\int_{\Rd}|x|^4\mu_t(dx)\le m.
  \end{equation}
\end{coro}
\begin{proof}
  The proof is similar to the proof of Corollary \ref{coro:hwi1}. The boundedness of moments gives tightness and we check easily that a limit point $\mu_\infty$ (for the weak topology) is a stationary solution.
Taking $(\rho_0, \rho_1) = (\mu_\infty, \mu_t)$ in HWI inequality (\ref{eq:hwi2}) and noticing that the minimum of power function in $W_2(\mu_t, \mu_\infty)$ (\ref{eq:hwi2}) is non-positive, we derive 
\begin{equation}
  C'\left(\Sigma(\mu_t)-\Sigma(\mu_\infty)\right)^{1+\frac{\eta}{\eta+2}} \le D(\mu_t)\le - \frac{d}{dt}\left[\Sigma(\mu_t)-\Sigma(\mu_\infty)\right],\ \forall t\ge 0,
\end{equation}
where $C'$ depends only on $ m, c$ and $\eta$. Integrating leads to the decay estimate (\ref{decay-estimate}). Taking $(\rho_0, \rho_1) = (\mu_t, \mu_\infty)$ in (\ref{eq:hwi2}) allows us to derive the transportation inequality (\ref{transpo}). Combining the two inequalities proves result (\ref{result-sqrt}).
  
\end{proof}

\section{Application to log gases}\label{section-5}
This section is devoted to the asymptotic behavior of uni-dimensional log gases. In all of the following, $V$ will denote a symmetric external potential and $W$ the logarithmic interaction $W = -\log{|\cdot|}$. Entropy (\ref{eq:entropy}) is now half of the free entropy introduced by Voiculescu in~\cite{voiculescu1993analogues}
\[ \Sigma(\mu) = \inv{2}\int_{\R}V(x)\mu(dx) - \inv{2}\int_{\R\times\R}\log{|x-y|}\mu(dx)\mu(dy).\]
Define the Hilbert transform of a measure $\mu\in\mathcal{M}$
  \[ H\mu(x) \equiv  -(W*\mu)'(x) =
      p.v.\int_{\R} \inv{x-y}\mu(dy).
      \]
where $p.v.$ denotes the principal value.

The granular media equation becomes the Fokker-Planck (\ref{eq:fokker_planck}), with the weak formulation:
\begin{equation}
\frac{d}{dt}\int_{\R} f(x)\mu_t(dx) = -\int_{\R} f'(x)\left(\inv{2}V'(x) - H\mu_t(x)\right)\mu_t(dx),\ \forall f\in \mathcal{C}_0^\infty(\R),
\end{equation}
or equivalently
\begin{equation}
\label{eq:fokker_planck_bis}
\frac{d}{dt}\int_{\R}f(x)\mu_t(dx) =\inv{2}\iint_{\R\times\R}\frac{f'(x) - f'(y)}{x - y}\mu_t(dx)\mu_t(dy) - \inv{2}\int_{\R} V'(x)f'(x)\mu_t(dx),\ \forall f\in \mathcal{C}_0^\infty(\R),
\end{equation}
In the rest of the paper, we say that $(\mu_t)_{t\ge0}\in C(\R_+, \mathcal{M})$, with initial data $\mu_0\in \mathcal{M}_2\cap L^{\infty}(\R)$ with finite entropy, is a solution of (\ref{eq:fokker_planck}) if (\ref{eq:fokker_planck_bis}) is satisfied and $\mu_t\in L^{\infty}(\R)$ for all $t > 0$.

Under the assumption
\begin{equation}
  \label{growth-V-log}
  \lim_{|x|\to\infty}V(x) - 2\log{|x|} = +\infty
\end{equation}
the entropy is lower-bounded and there exists a unique minimizer $\mu_V$ (see~\cite{saff2013logarithmic})
\begin{equation}
  \Sigma(\mu_V)\le \Sigma(\mu),\ \forall \mu\in\mathcal{M}.
\end{equation}
The entropy dissipation (\ref{eq:information}) is given by
\[D(\mu) = \int_\R \left|\inv{2}V'(x)-H\mu(x)\right|^2\mu(dx).\]

The existence of solutions and the dissipation property (\ref{eq:dissip}) has been essentially proved by Biane and Speicher in \cite{biane2001free} (Theorem 3.1 and Proposition 6.1, see also \cite{li2013generalized}) under the assumption that $V$ is $C^2$ and satisfies the growth assumption
\begin{equation}
  \label{growth-V}
  ax^2 + b \le \inv{2}xV'(x),\ \forall x\in \R,
\end{equation}
for some $a>0$ and $b\in\R$.

Combining conditions (\ref{growth-V-log}) and (\ref{growth-V}), we introduce the following assumptions ensuring the existence of a minimizer and the gradient flow property.
\begin{assump} There exist $a>0$ and $b\in\R$ such that
  \label{assump:growth}
  \begin{itemize}
  \item $\lim_{|x|\to\infty}V(x) - 2\log{|x|} = +\infty$,
  \item $ax^2 + b \le \inv{2}xV'(x),\ \forall x\in \R$.
  \end{itemize}
\end{assump}

This section is organized as follows. First, in order to prove Theorem \ref{thm:ramon}, we establish that $W=-\log{|\cdot|}$ satisfies a property similar to Assumption \ref{assumpD} (D2), which allows therefore to apply Theorems \ref{thm:hwi1} and \ref{thm:hwi2} to log gases. We provide then in Proposition \ref{prop:a_priori_est} uniform bounds for the moments of a solution to (\ref{eq:fokker_planck}). Those estimates are essential to prove tightness with respect to the Wasserstein distance and to bound uniformly the quantities $P_r$. We state and prove Theorem \ref{thm:log-alg} that gives algebraic convergence for any convex potential $V$. Finally, we consider quartic potential $V$ and prove in Theorems \ref{thm:confining} and \ref{thm:biane} exponential stability, when the parameters $c$ and $g$ are small enough in absolute value.

\subsection{HWI inequality}
Despite the singularity of $-\log{|\cdot|}$ at the origin and its non-convexity, we prove that $\mathcal{W}$ is convex-displacement and lower-bound explicitly the second order derivative. The case of log gases will then fall under the scope of application of Theorems \ref{thm:hwi1} and \ref{thm:hwi2}. Let first recall the following well known simple fact.
\begin{lem}
  \label{lem-opti-dim1}
  Let $\rho_0$ and $\rho_1$ be two measures on $\R$ with bounded positive densities. The optimal transport map $T$ carrying $\rho_0$ to $\rho_1$ is given by
  \begin{equation}
    \label{opti-transp}
    T = F_{\rho_1}^{-1}\circ F_{\rho_0}
  \end{equation}
  where $F_\rho$ denotes the CDF of measure $\rho$.
  Moreover, $T$ is derivable and for all $r>0$
  \begin{equation}
    \sup_{|x|\le r}T'(x) \le \f{\sup_{|x|\le r}\rho_0(x)}{\inf_{|x|\le r}\rho_1(T(x))} < \infty.
  \end{equation}
\end{lem}

We now prove Theorem~\ref{thm:ramon}.
\begin{proof}[Proof of Theorem~\ref{thm:ramon}]
  We first prove the result for measures with bounded positive densities. Let $\rho_0, \rho_1\in\mathcal{M}_2\cap L^{\infty}(\R)$ be as such. Assume either that $\rho_0$ and $\rho_1$ have the same center of mass or are both symmetric. Let $T$ be optimal transport map carrying $\rho_0$ to $\rho_1$
  \begin{equation}
    T = F^{-1}_{\rho_1}\circ F_{\rho_0}.
  \end{equation}
  and $(\rho_s)_{0\le s\le 1}$ the geodesic from $\rho_0$ to $\rho_1$
  \begin{equation}
    \rho_s = \left((1-s)Id + sT\right) \# \rho_0.
  \end{equation}
  Let $\epsilon\in \left(0, \inv{2}\right)$, $\eta > 0$ and $B_\eta = \{x\in\R : |x|\le \eta\}$.

  Using the monotonicity of $T$ and Lemma \ref{lem-opti-dim1} and setting $R(x, y) \equiv \f{T(x) - T(y)}{x-y}$, we have
  \begin{equation}
    \label{R-upper-bound}
    0\le \sup\limits_{\substack{x, y\in B_\eta\\ x\neq y}}R(x, y)\le \f{\sup\limits_{|x|\le \eta}\rho_0(x)}{\inf\limits_{|x|\le T(\eta)}\rho_1(x)} < \infty.
  \end{equation}
  Introduce
  \begin{equation}
    L_\epsilon(x, y)= \f{R(x, y) - 1}{1-\epsilon + \epsilon R(x, y)},
  \end{equation}
  such that for all $s\in (\epsilon, 1-\epsilon)$ and $x\neq y$
  \begin{equation}
    \left|\f{(1-s)(x - y) + s(T(x) - T(y))}{(1-\epsilon)(x-y) + \epsilon (T(x) - T(y))}\right| = 1 + (s-\epsilon)L_\epsilon(x, y).
  \end{equation}
  We see from the monotonicity of $T$ and bound (\ref{R-upper-bound}) that for all $s\in (\epsilon, 1-\epsilon)$
  \begin{equation}
    \label{L-lower-bound}
    1 + (s-\epsilon)L_\epsilon(x, y) = \f{1- s + sR(x, y)}{1-\epsilon + \epsilon R(x, y)} \ge \f{\epsilon}{1-\epsilon}\ge \epsilon > 0
  \end{equation}
  and that for all distinct $x$ and $y$ in $B_\eta$
  \begin{equation}
    \label{L-upper-bound}
    1 + (s-\epsilon)L_\epsilon(x, y) \le \f{1 + R(x, y)}{1-\epsilon}\le 2 + 2\sup\limits_{\substack{x, y\in B_\eta\\ x\neq y}}R(x, y) < \infty.
  \end{equation}

  Define
  \begin{equation}\label{omega-eps}
    \omega_\epsilon = -\inv{2}\iint L_\epsilon(x, y)\rho_0(dx)\rho_0(dy).
  \end{equation}
  
  Using the inequality $-\log{(1+x)} + x\ge 0$ for $x + 1> 0$, we deduce
  \begin{equation}
    \label{w-drop}
    \mathcal{W}(\rho_s) - \mathcal{W}(\rho_\epsilon) - (s-\epsilon)\omega_\epsilon \ge \inv{2}\iint\limits_{x, y\in B_\eta}[-\log{\left(1 + (s-\epsilon)L_\epsilon(x, y)\right)} + (s-\epsilon)L_\epsilon(x, y)]\rho_0(dx)\rho_0(dy).
  \end{equation}

  Define the functions
  \begin{equation}
    G_V : s\in [\epsilon, 1-\epsilon] \to \mathcal{V}(\rho_s) = \inv{2}\int V(x + s\theta(x))\rho_0(dx)
  \end{equation}
  and
  \begin{equation}
    G_W : s\in [\epsilon, 1-\epsilon] \to \inv{2}\iint_{x, y\in B_\eta}\left[-\log{\left(1 +  (s-\epsilon)L_\epsilon(x, y)\right)}\right]\rho_0(dx)\rho_0(dy).
  \end{equation}
  By the regularity of $V$, $G_V$ is twice derivable with
  \begin{equation}
    G_V'(\epsilon) = \inv{2}\int V'(x + \epsilon \theta(x))\theta(x)\rho_0(dx)
  \end{equation}
  and for all $s\in (\epsilon, 1-\epsilon)$
  \begin{equation}
    G_V''(s) = \inv{2}\int V''(x + s\theta(x))\theta(x)^2\rho_0(dx).
  \end{equation}

  By bounds (\ref{L-lower-bound}) and (\ref{L-upper-bound}), $G_W$ is twice derivable as well and its derivatives satisfy
  \begin{equation}
    G_W'(\epsilon) = - \inv{2}\iint_{x, y\in B_\eta} L_\epsilon(x, y) \rho_0(dx)\rho_0(dy)
  \end{equation}
  and for all $s\in (\epsilon, 1-\epsilon)$
  \begin{align}
    G_W''(s) &= \inv{2}\iint_{x, y\in B_\eta} \f{L_\epsilon(x, y)^2}{(1 + (s-\epsilon)L_\epsilon(x, y))^2}\rho_0(dx)\rho_0(dy).
    \end{align}

    Notice that for all $x\neq y$ and $s\in (\epsilon, 1 - \epsilon)$
    \begin{equation}
      \f{L_\epsilon(x, y)^2}{(1 + (s-\epsilon)L_\epsilon(x, y))^2} = \f{(R-1)^2}{(1-s + sR)^2}\le \f{R^2}{(1 - s + sR)^2} + \f{1}{(1-s+ sR)^2}\le \f{2}{\epsilon^2}.
    \end{equation}

    Therefore,
    \begin{equation}
   \iint_{x, y\notin B_\eta} \f{L_\epsilon(x, y)^2}{(1 + (s-\epsilon)L_\epsilon(x, y))^2}\rho_0(dx)\rho_0(dy)
   \le \f{2}{\epsilon^2}\left[\int_{|x|\ge \eta}\rho_0(dx)\right]^2.
    \end{equation}

    We deduce that for all $s\in (\epsilon, 1-\epsilon)$
    \begin{equation}
      G_W''(s)\ge \inv{2}\iint\f{(\theta(x) - \theta(y))^2}{(x - y + s(\theta(x) - \theta(y)))^2}\rho_0(dx)\rho_0(dy) -  \f{1}{\epsilon^2}\left[\int_{|x|\ge \eta}\rho_0(dx)\right]^{2}.
    \end{equation}

  Using (\ref{w-drop}) and Taylor's formula, we deduce for $u\in (\epsilon, 1-\epsilon)$ such that
  \begin{equation}
    \begin{split}
      \Sigma(\rho_{1-\epsilon}) - \Sigma(\rho_\epsilon) - (1-2\epsilon)(G_V'(\epsilon) + \omega_\epsilon)&\ge (G_V + G_W)(1-\epsilon) - (G_V+G_W)(\epsilon)\\
                                                                                                         &- (1-2\epsilon)(G_V+G_W)'(\epsilon)\\
                                                                                                         &= \inv{2}(G_V + G_W)''(u)(1-2\epsilon)^2
    \end{split}
  \end{equation}
  where 
  \begin{equation}
    \label{eq:ramon-0}
    \begin{split}
    (G_V+G_W)''(u) \ge \inv{2}\int V''(x + u\theta(x))\theta(x)^2\rho_0(dx)
&+ \inv{2}\iint\f{(\theta(x) - \theta(y))^2}{(x - y + u(\theta(x) - \theta(y)))^2}\rho_0(dx)\rho_0(dy)\\ &-  \f{1}{\epsilon^2}\left[\int_{|x|\ge \eta}\rho_0(dx)\right]^{2}.
    \end{split}
  \end{equation}

  From there, we follow the computations of the proof of Theorem \ref{thm:hwi1}. Firstly notice that
  
    \begin{align}
      \lim_{\epsilon\to 0}[G_V'(\epsilon) + \omega_\epsilon] &= \inv{2}\int V'(x)\theta(x)\rho_0(dx) - \inv{2}\int [R(x, y) - 1]\rho_0(dx)\rho_0(dy)\\
                                 &= \int\left(\inv{2}V'(x) - H\rho_0(x)\right)\rho_0(dx)\\
                                 &\ge -\sqrt{D(\rho_0)}W_2(\rho_0, \rho_1).
    \end{align}

  Secondly, taking $\gamma = \inv{4r^{2}}$ and following (\ref{eq:taylor_2_1}) - (\ref{eq:hwi_proof}), we obtain
  \begin{equation}
    \label{eq:ramon-1}
    \inv{2}\int V''(x + u\theta(x))\theta(x)^2\rho_0(dx)
    + \inv{2}\iint\f{(\theta(x) - \theta(y))^2}{(x - y + u(\theta(x) - \theta(y)))^2}\rho_0(dx)\rho_0(dy)\ge \lambda_r W_2(\rho_0, \rho_{1-\epsilon})^2,
  \end{equation}
  where $\lambda_r$ is given either by (\ref{lambda-fcm}) or (\ref{lambda-sym}).

  Notice that this estimate is independent of $u\in (\epsilon, 1-\epsilon)$.
  Recalling that $W_2(\rho_0, \rho_{1-\epsilon}) = (1-\epsilon)W_2(\rho_0, \rho_1)$, we conclude that for all $\eta >0$ 
  \begin{equation}
    \label{eq:ramon}
    \Sigma(\rho_{1-\epsilon}) - \Sigma(\rho_\epsilon) \ge (1-2\epsilon)[G'_V(\epsilon) + \omega_\epsilon] + \lambda_r(1-\epsilon)^2W_2(\rho_0, \rho_1)^2 - \f{1}{\epsilon^2}\left[\int_{|x|\ge \eta}\rho_0(dx)\right]^{2}.
  \end{equation}
  
  From there, do successively $\eta\to \infty$ and $\epsilon\to 0$ to derive the result. 
  If $\rho_0$ and $\rho_1$ do not have positive densities, apply (\ref{eq:hwi}) for the sequences
  \begin{equation}
    \rho^\delta_i(dx) = \int \f{e^{-(x-y)^2/(2\delta)}}{\sqrt{2\pi\delta}}\rho_i(dy),\ \delta > 0,\ i = 0, 1,
  \end{equation}
  and let $\delta\to 0$. Notice that $\rho^\delta_i$ has a positive density and belong to  $L^{\infty}(\R)$ with $|\rho^{\delta}_i|_{L^{\infty}(\R)}\le \inv{\sqrt{2\pi\delta}}$, and that the map $\rho_i\mapsto \rho_i^{\delta}$ leaves the center of mass or the symmetry of the measure $\rho_i$ invariant. By standard arguments (in particular using the isometry property of the Hilbert transform and the fact that $\rho_i\in \mathcal{M}(\R)\cap L^{\infty}(\R)\subset L^2(\R)$) we show that $\Sigma(\rho_i) < \infty $, $\lim_{\delta\to 0}\Sigma(\rho_i^{\delta}) = \Sigma(\rho_i)$, $\lim_{\delta\to0}W_2(\rho_0^{\delta}, \rho_1^{\delta}) = W_2(\rho_0, \rho_1)$ and $\lim_{\delta\to 0}D(\rho^{\delta}_0) = D(\rho_{0})$.
\end{proof}
\begin{rmk}\label{rmk:ledoux}
  Our theorem is similar to Theorem 1.4~\cite{li2013generalized}, but we do not assume any assumption regarding the compactness of the support. Therefore, our proof can be extended to gases with a diffusive internal energy (like in \cite{tugaut2013convergence} for example).
 In Theorem 5~\cite{Ledoux_2009}, the authors established a HWI inequality for log gases. Their result can recovered by considering a simpler version of estimate (\ref{w-drop}). Indeed, by applying $-\log{(1 + x)} + x\ge 0$ on the whole space, we find
  \begin{equation}
    \mathcal{W}(\rho_s) -\mathcal{W}(\rho_0) -(s-\epsilon)\omega_\epsilon\ge 0.
  \end{equation}
  Using this crude estimate, (\ref{eq:ramon}) becomes
  \begin{equation}
    \Sigma(\rho_{1-\epsilon}) - \Sigma(\rho_\epsilon) \ge (1-2\epsilon)\left[G'_V(\epsilon)+\omega_\epsilon\right] + \inv{2}\inf_{x\in \R}V''(x)(1-\epsilon)^2W_2(\rho_0, \rho_1)^2.
  \end{equation}
  Therefore under the assumption that $\inf_{x\in \R}V''(x)\ge 2\lambda > 0$, letting $\epsilon\to 0$, we find 
  \begin{equation}
    \label{HWI-ledoux}
    \Sigma(\rho_0) - \Sigma(\rho_1)\le \sqrt{D(\rho_0)}W_2(\rho_0, \rho_1) - \f{\lambda}{2} W_2(\rho_0, \rho_1)^2.
  \end{equation}
\end{rmk}

\subsection{Moment estimates}
We establish upper-bounds for the moments of solutions of (\ref{eq:fokker_planck}). Those estimates are useful to prove tightness of a solution and to bound uniformly in $t$ the tail probabilities $\Prob_{\mu_t}(|x|\ge r)$. The idea is essentially to use dynamics (\ref{eq:fokker_planck_bis}) to get a differential inequation satisfied by the moments.

Let's start with the following lemma, justifying extension of test functions to power functions. We delay the proof to the appendix.
\begin{lem}
\label{lem:power_testf}
Let $(\mu_t)_{t\ge 0}$  be a solution of (\ref{eq:fokker_planck}) such that for all $t\ge 0$, $\mu_t$ has a finite moment of order $p\ge 1$. Assume that
\begin{enumerate}
    \item $s\mapsto\int_{\R} (1+|V'(x)|)|x|^p\mu_s(dx)\in L_{loc}^1([0, \infty[),$
    \item $x\mapsto |V'(x)||x|^p\in L^1(\R,\mu_s)$ for all $s\ge 0$.
\end{enumerate}
Then the power function $x\in\R\rightarrow |x|^p$ is a valid test function in (\ref{eq:fokker_planck_bis}).
\end{lem}

Using this lemma, we prove 
\begin{prop}
\label{prop:a_priori_est}
  Let $(\mu_t)_{t\ge 0}$ be a solution of (\ref{eq:fokker_planck}) with finite moments up to order $p+2$ for some $p\ge 2$. Then, there exists $M_p>0$ such that 
  \[\sup_{t\ge 0}\int_{\R} |x|^{p}\mu_t(dx)\le M_{p}.\]
\end{prop}

\begin{proof}
  Denote $m_p(t) = \int_{\R} |x|^{p}\mu_t(dx)$ for $t, p \ge 0$.
  According to Lemma \ref{lem:power_testf}, we can apply (\ref{eq:fokker_planck_bis}) with the test function $f:x\mapsto |x|^{p}$ so as to obtain
  \begin{equation}
  \label{eq:ineqmp}
                \dot{m}_{p}(t) \le \inv{2}\int_{\R\times\R}\frac{|f'(x) - f'(y)|}{|x - y|}\mu_t(dx)\mu_t(dy)-pam_{p}(t)+p|b|pm_{p-2}(t).
  \end{equation}

    Without loss of generality, we may assume that $|x| < |y|$. We see easily that
    \[
        \frac{|f'(x) - f'(y)|}{|x - y|} \le p\frac{|x|x|^{p-2} - y|y|^{p-2}|}{|x-y|} \le p(|x|^{p-2}+|y|^{p-2}) + p|x|\frac{||x|^{p-2}-|y|^{p-2}||}{||x|-|y||}.
      \]
      We check easily that if $p \ge 3$, then the previous inequality implies
      \[ \frac{|f'(x) - f'(y)|}{|x - y|} \le p(|x|^{p-2}+|y|^{p-2}) + p(p-2)(|x|^{p-2} + |y|^{p-2})  \]
      and if $2 \le p < 3$
      \[ \frac{|f'(x) - f'(y)|}{|x - y|} \le p(|x|^{p-2}+|y|^{p-2}) + p(p-2)|x|^{p-2}.\]
      In both cases for all $x, y\in\R$
      \begin{equation}
      \label{eq:kestim}
          \frac{|f'(x) - f'(y)|}{|x - y|}\le p(|x|^{p-2}+|y|^{p-2}) + p(p-2)(|x|^{p-2} + |y|^{p-2}).
      \end{equation}
      
Using the fact that $m_l(t) \le m_{p}^{l/p}(t)$ for $l \le p$, and inserting (\ref{eq:kestim}) in (\ref{eq:ineqmp}), we deduce the ordinary differential inequation
\begin{equation}
    \dot{m}_{p}(t)\le -pam_{p}(t) + p\left(|b|+p-1\right)m_{p}^{(p-2)/(p)}(t),\ t\ge 0.
\end{equation}
This inequality can be written as
\begin{equation}
    \label{eq:inequation}
    \dot{m}_{p}(t)\le Q(m_{p}(t)),\ t\ge 0
\end{equation}
where $Q(x) = -pax + p(\left |b| + p-1\right)x^{(p-2)/p}$ satisfies $Q(x)\underset{+\infty}{\sim}-pax$.

We will show that this inequality implies the uniform boundedness of moments. Consider $M$ defined by:
\[ M = \sup\{x\ge 0: Q(x)\ge 0 \} = \left(\f{|b| + p-1}{a}\right)^{p/2}.\]
Notice that $m_{p}$ is non-increasing for $m_{p}> M$. Set $M_{p}\equiv \max(M, m_{p}(0))$ and $T = \{t\ge 0: m_{p}(t)\le M\}$. Without loss of generality we may assume that $0\in T$, otherwise $m_{p}$ is decreasing until $t\in T$. The set $\R_+\setminus T$ is open (in $\R_+$) so can be written as
\[\R_+\setminus T = \bigcup_i\ ]s_i, t_i[.\]

On the one hand
\[ \dot{m}_{p}(t)< 0 \textrm{ for } s_i<t<t_i\]
and on the other hand
\[ m_{p}(s_i) = m_{p}(t_i) = M.\]
Those two contradictory statements imply that $T = \R_+$ and for all $t\ge 0$:
\[ m_{p}(t) \le M_{p}.\]
This achieves the proof.
\end{proof}
\begin{rmk}
  In the next section, we consider quartic potential $V$ and we derive in the same way Proposition~\ref{prop:moment-2-quartic} providing uniform bounds for the second order moments of a solution in a more precise way. Still, we would like to stress the usefulness of the more general Proposition~\ref{prop:a_priori_est}, which allows to show the stability of any solution with finite moments of order $p + 2 \ge 2$ towards a stationary measure with respect to the Wasserstein distance of order $p$. Additionally, we will use this proposition for the proof of Theorem~\ref{thm:log-alg}.
\end{rmk}




  

\subsection{Stability for convex potentials}
From the proofs of Theorems~\ref{thm:hwi2} and~\ref{thm:ramon}, we derive Theorem~\ref{thm:log-alg}.
\begin{proof}[Proof of Theorem~\ref{thm:log-alg}]
  Fix $\epsilon\in (0, 1/2)$ and $r>0$. Let $\rho_0, \rho_1\in\mathcal{M}_2$ and $T$ be the optimal transport map from $\rho_0$ to $\rho_1$. Set $\theta = T - Id$. Denote $(\rho_s)_{s\in [0, 1]}$ the geodesic between $\rho_0$ and $\rho_1$.
  Following to the proof of Theorem~\ref{thm:ramon}, and using that the fact that$D^2V\ge 0$, there exists $u\in (\epsilon, 1-\epsilon)$ such that
  \begin{equation}
    \label{hwi-ramon-alg}
    \begin{split}
      \Sigma(\rho_{1-\epsilon}) -\Sigma(\rho_{\epsilon})\ge (1-2\epsilon)[G_V'(\epsilon) + \omega_\epsilon]
      &+ 
   \inv{2}\iint\f{(\theta(x) - \theta(y))^2}{(x - y + u(\theta(x) - \theta(y)))^2}\rho_0(dx)\rho_0(dy)\\
 &- \f{1}{\epsilon^2}\left[\int_{|x|\ge \eta}\rho_0(dx)\right]^{2},
    \end{split}
  \end{equation}
  where $\omega_\epsilon$ is defined as in (\ref{omega-eps}).

  From there, follow the proof of Theorem~\ref{thm:hwi2} by taking $c = 1$ and $\eta = 2$. Estimate (\ref{eq:hwi2}) becomes
  \begin{equation}
    \inv{2}\iint\f{(\theta(x) - \theta(y))^2}{(x - y + u(\theta(x) - \theta(y)))^2}\rho_0(dx)\rho_0(dy)
    \ge C W_2(\rho_0, \rho_1)^{4}.
  \end{equation}

  Inserting this estimate in (\ref{hwi-ramon-alg}) and then letting $\eta\to\infty$ and $\epsilon\to 0$, we get 
  \begin{equation}
    \label{eq:alg-hwi}
    \Sigma(\rho_1) - \Sigma(\rho_0)\ge -\sqrt{D(\rho_0)}W_2(\rho_0, \rho_1) + CW_2(\rho_0, \rho_1)^4.
  \end{equation}
  The result follows by mimicking the proof of Corollary~\ref{coro:hwi2} and using the tightness of $(\mu_t)_{t\ge 0}$ ensured by Proposition~\ref{prop:a_priori_est}.
\end{proof}

\subsection{Stability for confining quartic potential}
In this subsection, we consider the confining quartic potential $V(x) = \f{x^4}{4} + c\f{x^2}{2}$. We recall the following results:
\begin{itemize}
    \item If $c\ge 0$, then $V$ satisfies Ledoux's assumptions \cite{Ledoux_2009}, which allows to prove exponential convergence towards the equilibrium measure $\mu_V$ \cite{li2013generalized}, which density is given by
    \begin{equation}
        \label{eq:density_mu_v}
        \mu_V(dx) = \frac{1}{\pi}\left(\frac{1}{2}x^2 + b\right)\sqrt{a^2-x^2}1_{[a, -a]}(x),
    \end{equation}
    where
    \[a^2=\frac{\sqrt{4c^2+48}-2c}{3},\ \ b = \frac{c + \sqrt{\frac{c^2}{4} + 3}}{3}. \]
    \item If $-2 < c < 0$, $\mu_V$ is the unique stationary measure, and stability is proved in \cite{donati2018convergence}. The density is again given by (\ref{eq:density_mu_v}).
    \item If $c < - 2$, Donati-Martin et al. have proved in \cite{donati2018convergence} that there could be multiple stationary measures.
\end{itemize}

The main result of this subsection will follow from Theorem \ref{thm:ramon} applied to quartic potential $V$ with $c$ negative and close enough to zero. 

The same considerations as in the proof of Proposition \ref{prop:a_priori_est} lead to the following estimate of second order moment.
\begin{prop}
  \label{prop:moment-2-quartic}
  Let $(\mu_t)_{t\ge 0}$ be a solution of (\ref{eq:fokker_planck}) with quartic $V$ $(C1)$. Assume that $(\mu_t)_{t\ge 0}$ have finite fourth moments for all times $t\ge 0$. Then
  \begin{equation}
    \sup_{t\ge 0} \int_{\R}x^2\mu_t(dx) \le \max\left(\int_{\R}x^2\mu_0(dx), \frac{-c + \sqrt{c^2+4}}{2}\right).
  \end{equation}
\end{prop}

We are now ready to prove the Theorem \ref{thm:confining}.
\begin{proof}[Proof of Theorem \ref{thm:confining}]
Let $\mu_0$ satisfying the moment condition (\ref{moment-cond}) and $(\mu_t)_{t\ge 0}$ be a solution of (\ref{eq:fokker_planck}).
  According to Proposition~\ref{prop:a_priori_est}, the second order moments of $(\mu_t)_{t\ge 0}$ are uniformly bounded in time. $\Sigma$ is lower-bounded and $D$ is lower-semi-continuous, so similarly to the proof of Corollary~\ref{coro:hwi1}, $(\mu_t)_{t\ge 0}$ weakly converges towards a stationary solution $\mu_\infty$. According to Proposition 2.7~\cite{donati2018convergence}, the unique stationary solution for $c\in [-2, 0)$ is the minimizer of the entropy $\mu_V$. Finally, the uniform bounds of the moments give tightness and convergence with respect to the Wasserstein distance.

   Assume that $(\mu_t)_{t\ge 0}$ has a fixed center of mass. According to Theorem~\ref{thm:ramon}, for all measures $\rho_0,\rho_{1}\in\mathcal{M}_2$ with finite entropy
  \begin{equation}
    \label{eq:hwi_log}
    \Sigma(\rho_0) - \Sigma(\rho_1) \le W_2(\rho_0, \rho_1)\sqrt{D(\rho_0)} - \f{\lambda}{2} W_2(\rho_0, \rho_1)^2,
  \end{equation}
  with constants $(\alpha, \beta, \gamma)$ given by
  \begin{equation}
    \begin{cases}
      \alpha = 3r^2 + c,\\
      \beta = - c,\\
      \gamma = \inv{4r^2},
    \end{cases}
  \end{equation}
  and the optimal rate
  \begin{equation}
    \label{const-lambda}
    \lambda = \f{c}{2} + \sup_{r > 0}\left[\f{\min\left(3r^2, \inv{2r^2}\right)}{2} - \f{P_r}{2r^2}\right].
  \end{equation}
  Under the assumption that
  \begin{equation}
    \int_{\R}x^2\mu_0(dx) \le \f{-c+\sqrt{c^2+4}}{2}
  \end{equation}
  we have for $c\in [-2, 0)$
  \begin{equation}
    P_r\le\f{-c+\sqrt{c^2+4}}{2r^2}. 
  \end{equation}
  It follows that the constant $\lambda$ (\ref{const-lambda}) can be lower-bounded by
  \begin{equation}
    \lambda \ge \f{c}{2} + \inv{4}\sup_{r>0}\inv{r^4}\left[\min\left(6r^6, r^2\right) - (- c + \sqrt{c^2+4})\right].
  \end{equation}

  We solve this optimization easily and find
  \begin{equation}
    \lambda \ge \inv{16(-c + \sqrt{c^2 + 4})} + \f{c}{2}.
  \end{equation}
  We check that this last quantity is positive for any $c\in (c^*, 0)$ with
  \begin{equation}
    c^* = -\inv{4\sqrt{17}}.
  \end{equation}
  From there, an obvious variant of the reasoning of Corollary \ref{coro:hwi1} makes it possible to conclude.

  In the case of $\mu_0$ symmetric, the rate $\lambda$ can be slightly improved. Indeed, we can improve estimate (\ref{estimate-rho-s}) in the following way. Write for $\rho_0$ and $\rho_1$ symmetric
  \begin{equation}
    \left((1-s)x + sT(x)\right)^2 = (1-s)^2x^2 + s^2T(x) + 2s(1-s)xT(x).
  \end{equation}
  Now, $T$ being odd and the gradient of a convex function, we have $T(0) = 0$ and
  \begin{equation}
    xT(x)\ge 0.
  \end{equation}
  Therefore
  \begin{equation}
    \left((1-s)x + sT(x)\right)^2 \ge (1-s)^2x^2 + s^2T(x).
  \end{equation}
  We deduce that if $s\le \inv{2}$
  \begin{equation}
    |(1-s)x + sT(x)|\le r\implies |x|\le 2r,
  \end{equation}
  and if $s\ge\inv{2}$
  \begin{equation}
    |(1-s)x + sT(x)|\le r\implies |T(x)|\le 2r.
  \end{equation}
  Therefore
  \begin{equation}
    \int_{|x|\le r}\rho_s(dx)\ge 1_{s\le 1/2}\int_{|x|\le 2r}\rho_0(dx) + 1_{s\ge 1/2}\int_{|x|\le 2r}\rho_1(dx)\ge 1 - P_{2r}.
  \end{equation}
  This bound is better because $P_{2r}\le 2P_r$.

  Consequently, in the case of $\mu_0$ symmetric, the optimal $\lambda$ is given by
  \begin{equation}
    \label{prblm-lambda}
    \lambda = \inv{4}\sup_{r>0}\inv{r^2}\min\left(6r^4, 1-P_{2r}\right) + \f{c}{2}.
  \end{equation}

  Similarly, $\lambda$ can be lower-bounded by
  \begin{equation*}
    \begin{split}
  \lambda &\ge \inv{4}\sup_{r>0}\inv{r^4}\min{\left(6r^6, r^2-\f{(-c+\sqrt{c^2+4})}{8}\right)} + \f{c}{2}\\
          &=\inv{2(-c+\sqrt{c^2+4})} + \f{c}{2}.
    \end{split}
  \end{equation*}
  This last quantity is positive for any $c\in (c^*, 0)$ with
  \begin{equation}
    c^* = -\inv{\sqrt{6}}.
  \end{equation}
  The result follows.
\end{proof}

Finally, we end this section by giving the proof of Theorem \ref{thm:bonus}, which overrides the assumptions of a fixed center of mass or symmetry. The optimal $c^*$ obtained is numerically much smaller than the constant obtained in the previous theorem. Therefore, the value of our result lies in its proof. We first need two technical lemmas. The first one is proved in the appendix, the second proof is omitted and follows exactly the proof of Proposition \ref{prop:a_priori_est}.

\begin{lem}
  \label{lem:tail-rho-s}
  Let $(\rho_s)_{0\le s\le 1}=(((1-s)Id + sT)\#\rho_0)_{0\le s\le 1}$ be a geodesic between $\rho_0\in\mathcal{M}_4$ and a symmetric measure $\rho_1\in\mathcal{M}_4$. Then for all $s\in [0, 1]$ and $r>0$
  \begin{equation}
    \int\limits_{|x|\le r}x^2\rho_s(dx) \ge \inv{2}\min\left(\int x^2\rho_0(dx), \int x^2\rho_1(dx)\right) - \f{8}{r^2}\max\left(\int x^4\rho_0(dx), \int x^4\rho_1(dx)\right).
  \end{equation}
\end{lem}

\begin{lem}
  Let $(\mu_t)_{t\ge 0}$ be a solution of (\ref{eq:fokker_planck}) with quartic $V$. Assume that $(\mu_t)_{t\ge 0}$ have finite sixth moments for all times $t\ge 0$. Then
  \begin{equation}
    \inf_{t> 0} \int x^2\mu_t(dx) \ge \min\left(\int x^2\mu_0(dx), \left(\f{2}{-c+\sqrt{c^2+16}}\right)^4\right)
  \end{equation}
  and
  \begin{equation}
    \sup_{t\ge 0} \int x^4\mu_t(dx) \le \max\left(\int x^4\mu_0(dx), \left(\f{-c+\sqrt{c^2+12}}{2}\right)^2\right).
  \end{equation}
\end{lem}

\begin{proof}[Proof of Theorem \ref{thm:bonus}]
  As before, tightness is clear. It is enough to show a HWI inequality with rate $\lambda>0$.
  We use the notations of the proof of Theorem \ref{thm:ramon}. In the following $\rho_0$ and $\rho_1$ denote respectively $\mu_t$, $t\ge 0$ and $\mu_V$.
  Start with (\ref{eq:ramon-0}) 
  \begin{equation}
    \begin{split}
    (G_V+G_W)''(u) \ge \inv{2}\int V''(x + u\theta(x))\theta(x)^2\rho_0(dx)
+ &\inv{2}\iint\f{(\theta(x) - \theta(y))^2}{(x - y + u(\theta(x) - \theta(y)))^2}\rho_0(dx)\rho_0(dy)\\ -  &\f{1}{\epsilon^2}\left[\int_{|x|\ge \eta}\rho_0(dx)\right]^{2}.
    \end{split}
  \end{equation}

  Therefore, denoting $\psi(x) = 3x^2$, plugging $V''(x + u\theta(x)) = \psi(x + u\theta(x)) + c$ and splitting the integrals according to $|x+u\theta(x)|\le r$, we have
  \begin{equation}
    \begin{split}
    (G_V+G_W)''(u) \ge &\inv{2}\int_{|x+u\theta(x)|\le r} \psi(x+u\theta(x))\theta(x)^2\rho_0(dx) + \f{3r^2}{2}\int_{|x+u\theta(x)|> r} \theta(x)^2\rho_0(dx)\\
      + &\inv{8r^2}\iint\limits_{\substack{|x+u\theta(x)|\le r\\|y+u\theta(y)|\le r}}(\theta(x) - \theta(y))^2\rho_0(dx)\rho_0(dy)\\
      +&\f{c}{2}W_2(\rho_0, \rho_1)^2 -  \f{1}{\epsilon^2}\int_{|x|\ge \eta}\rho_0(dx).
    \end{split}
  \end{equation}

  We now adapt an idea used in section 4.5 \cite{carrillo2003kinetic} to our specific integrals. Write
  \begin{equation}
    \begin{split}
      \iint\limits_{\substack{|x+u\theta(x)|\le r|y+u\theta(y)|\le r}}(\theta(x) - \theta(y))^2\rho_0(dx)\rho_0(dy) &= 2\int\limits_{|x+u\theta(x)|\le r}\rho_0(dx)\int_{|x+u\theta(x)|\le r}\theta(x)^2\rho_0(dx)\\
      &-2\left(\int_{|x+u\theta(x)|\le r}\theta(x)\rho_0(dx)\right)^2.
    \end{split}
  \end{equation}

  From
  \begin{equation}
    \begin{split}
      \left(\int\limits_{|x+u\theta(x)|\le r}\theta(x)\rho_0(dx)\right)^2\le &\int\limits_{|x+u\theta(x)|\le r}\inv{1 + 2r^2\psi(x + u\theta(x))}\rho_0(dx)\\
      \times&\int\limits_{|x+u\theta(x)|\le r}(1 + 2r^2\psi(x + u\theta(x)))\theta(x)^2\rho_0(dx)
    \end{split}
  \end{equation}
  we derive
  \begin{equation}
    \begin{split}
      (G_V+G_W)''(u) &\ge
                       \inv{4r^2}\left(\int\limits_{|x+u\theta(x)|\le r}\left[1 -\inv{1 + 2r^2\psi(x + u\theta(x))}\right]\rho_0(dx)\right)\\
                       \times&\int\limits_{|x+u\theta(x)|\le r}(1 + 2r^2\psi(x+u\theta(x)))\theta(x)^2\rho_0(dx)\\
      &+ \f{3r^2}{2}\int_{|x+u\theta(x)|> r} \theta(x)^2\rho_0(dx) +\f{c}{2}W_2(\rho_0, \rho_1)^2 -  \f{1}{\epsilon^2}\int_{|x|\ge \eta}\rho_0(dx).
    \end{split}
  \end{equation}

  The goal is to lower-bound the term
  \begin{equation}
    \int\limits_{|x+u\theta(x)|\le r}\left[1 -\inv{1 + 2r^2\psi(x + u\theta(x))}\right]\rho_0(dx) = \int_{|x|\le r}\f{2r^2x^2}{1+2r^2x^2}\rho_u(dx).
  \end{equation}

  In \cite{carrillo2003kinetic}, a similar quantity appears and the authors prove that it is bounded away from zero by using the internal energy. In absence of any internal energy, we have to exploit the logarithmic repulsive interaction to show that $\rho_s$ - the interpolation between $\mu_V$ and $\mu_t$ - cannot be concentrated around the origin. This intuition is implemented by looking at the second order moments.
  
  Set $m = \left(\f{2}{-c+\sqrt{c^2+16}}\right)^4$ and $M = \left(\f{-c+\sqrt{c^2+12}}{2}\right)^2$.
  By Lemma \ref{lem:tail-rho-s} ($\rho_1= \mu_V$ is indeed symmetric), we have
  \begin{equation}
\int_{|x|\le r}\f{2r^2x^2}{1+2r^2x^2}\rho_u(dx)\ge \f{2r^2}{1+2r^4}\left[\inv{2}m - \f{8}{r^2}M\right].
  \end{equation}

  Use this estimate and $1 + 2r^2\psi(x + u\theta(x))\ge 1$ to deduce
  \begin{equation}
      (G_V+G_W)''(u) \ge \lambda_r W_2(\rho_0, \rho_1)^2 -  \f{1}{\epsilon^2}\left[\int_{|x|\ge \eta}\rho_0(dx)\right]^{2}
  \end{equation}
  where $\lambda_r$ is given by
  \begin{equation}
    \lambda_r = \f{c}{2} + \inv{2}\min\left(3r^2, \inv{1+2r^4}\left[\inv{2}m - \f{8}{r^2}M\right]
 \right).
  \end{equation}
  
  From this point, follow the end of the proof of Theorem \ref{thm:ramon} to conclude
  \begin{equation}
    \Sigma(\rho_0) - \Sigma(\rho_1) \le W_2(\rho_0, \rho_1)\sqrt{D(\rho_0)} - \f{\lambda_r}{2} W_2(\rho_0, \rho_1)^2.
  \end{equation}
   We see clearly that $\sup_{r>0}\lambda_{r}$ can be made positive for $c$ negative and close enough to zero. The optimal rate can be lower-bounded by
  \begin{equation}
    \lambda = \sup_{r> 0}\lambda_r \ge \f{c}{2} + \inv{4}\sup_{r^2\ge 16M/m}\left[\inv{r^4}\left(\f{m}{2} - \f{8M}{r^2}\right)\right] = \f{c}{2} + \f{m^3}{13924M^2}.
  \end{equation}
  Numerically, we find that $\lambda >0$ for $c > -3.00\times 10^{-9}$.

\end{proof}


\subsection{Stability for non-confining quartic potentials}

The difficulty of the non-confining quartic potential
\begin{equation}
    \label{eq:non_conf_V}
    V(x) = g\frac{x^4}{4} + \frac{x^2}{2},\ g<0,
\end{equation}
is twofold: how do we define a solution of (\ref{eq:fokker_planck}) and what are the equilibrium and stationary measures? Indeed, solutions of (\ref{eq:fokker_planck}) may explode and Assumption \ref{assump:growth} is not satisfied by this kind of potential, so we do not necessarily have the existence of an equilibrium measure. In \cite{allez2015random}, authors give insights on how to restart solutions after explosion time. This approach is different from the idea of Biane and Speicher in section 7.1 \cite{biane2001free}, where the authors explain how to define bounded solutions of (\ref{eq:fokker_planck}), which stay around the origin, and exhibit a good candidate for the equilibrium measure. In the following, we will adopt this last point of view.

Let $(\mathcal{A}, \tau,(\mathcal{A}_t)_{t\ge 0},(S_t)_{t\ge 0})$ be a filtered non-commutative probability space with a free Brownian motion $(S_t)_{t\ge 0}$. Denote $\mathcal{A}^{op}$ the algebra $\mathcal{A}$ with operation $a\cdot_{\mathcal{A}^{op}}b = b\cdot_{\mathcal{A}} a$ for $a, b\in \mathcal{A}$. See \cite{biane1998free}, \cite{biane1997free} and \cite{biane2001free} for more details on free probability and free stochastic processes. For all $t>0$, denote $\rho_t$ the density of semi-circular distribution of mean zero and variance $t$. In those conditions, $S_t$ has a distribution given by $\rho_t$.

\begin{equation}
  \rho_t(dx) = \f{2}{\pi t}\sqrt{t - x^2}1_{[-\sqrt{t}, \sqrt{t}]}(x)dx,\ \forall t> 0.
\end{equation}

In the framework of free probabilities, a solution $(\mu_t)_{t\ge 0}$ of (\ref{eq:fokker_planck}) can be seen as the distribution of the free stochastic process $(X_t)_{t\ge 0}$, solution of the free stochastic differential equation
\begin{equation}
    \label{eq:free_fokker_planck}
    \begin{cases}
    X_t = X_0 + S_t - \frac{1}{2}\int_0^tV'(X_t)dt\\
      Law(X_0) = \mu_0
    \end{cases}
\end{equation} 
We shall denote $\mu_t = Law(X_t)$. We now recall the free It\^o formula and the free Burkholder-Gundy inequality. For any polynomial $\phi = \sum_{n\ge 0}\phi_n X^n$, denote for $X\in \mathcal{A}$  the element $\partial\phi(X)$ of $\mathcal{A}\otimes\mathcal{A}^{op}$ defined by
\begin{equation}
  \partial\phi(X) = \sum_{n\ge 0} \phi_n\sum_{k=0}^{n-1}X^k\otimes X^{n-k-1}.
\end{equation}
Introduce the operator $\Delta_t$ defined by
\begin{equation}
  \Delta_t \phi(x) = 2\frac{d}{dx}\left(\int_\R\f{\phi(x) - \phi(y)}{x-y}\rho_t(dy)\right).
\end{equation}

The free It\^o's formula is then written
\begin{equation}
  \phi(X_t) = \phi(X_0) + \int_0^t\partial \phi(X_s)\sharp dX_s + \inv{2}\int_0^t\Delta_s\phi(X_s)ds,\ \forall t\ge 0,
\end{equation}
where $\int_0^t\partial\phi(X_s)\sharp dX_s$ denotes the free stochastic integral of the bi-process $(\partial\phi(X_{s}))_{s\ge 0}$ with respect to the free It\^o process $(X_t)_{t\ge 0}$.

The free Burkholder-Gundy inequality states that
\begin{equation}
  \left\|\int_0^tY_s\sharp dS_s\right\|\le 2\sqrt{2}\sqrt{\int_0^t\|Y_s\|^2ds},\ \forall t\ge 0,
\end{equation}
for any free It\^o process  $(Y_t)_{t\ge 0}$.

We are now ready to define solutions to (\ref{eq:free_fokker_planck}), detailing an idea of Biane and Speicher succinctly presented in \cite{biane2001free}.
\begin{prop}
  \label{prop:biane}
  Let $g\in \left(-\inv{81+ 36\sqrt{5}} , 0\right)$ and $\mu_0$ be an initial with support included in $(-m(g), m(g))$, with $m(g)$ given by
  \begin{equation}
    m(g) = \sqrt{- \inv{3g} - \f{4}{\sqrt{-g}} - 3}.
  \end{equation}
  Then the process $(X_t)_{t\ge 0}$ defined by (\ref{eq:free_fokker_planck}) exists for all times and the support of $\mu_{t}$ remains in a set of strict convexity of $V$. More precisely, if for some $m\in [0, m(g))$
  \begin{equation}
    supp(\mu_0)\subset [-m, m]
  \end{equation}
  then for all $t\ge 0$
  \begin{equation}
    \label{biane-support}
    supp(\mu_t)\subset \left[- \sqrt{m^2 + \f{4}{\sqrt{-g}} + 3},\sqrt{m^2 + \f{4}{\sqrt{-g}} + 3}\right],
  \end{equation}
  and
  \begin{equation}
    \label{biane-convexity-set}
     \inf_{x\in supp(\mu_t)}V''(x) = 1 + 3g\left(m^2 + \f{4}{\sqrt{-g}} + 3\right) > 0.
  \end{equation}
\end{prop}
\begin{proof}
  Let $\epsilon\in (0, 1)$.
  Set $h = \sqrt{\inv{-3g}}$ and $M = \sqrt{1-\epsilon}h$. With those notations
  \begin{equation}
  \inf_{|x|\le M} V''(x) = 1 + 3gM^2 = \epsilon >0.
  \end{equation}
  The idea of the proof is to show that the stopping time
  \[ T = \inf\{ t \ge 0: \|X_t\| > M \} \]
  is almost surely infinite.

  Fix $\delta > 0$. We apply free It\^o's formula to $e^{\delta t}G(X_t)$ with $G(x) = x^2$:
  \begin{equation}
    \begin{split}
      e^{\delta t}X_t^2 = X_0^2 + \int_0^te^{\delta s}\delta X_s^2ds + \int_0^te^{\delta s}1_{\mathcal{A}}ds &+ \int_0^te^{\delta s}\left(X_s\otimes 1_{\mathcal{A}} + 1_{\mathcal{A}}\otimes X_s\right)\#dS_s\\
      &- \inv{2}\int_0^te^{\delta s}X_sV'(X_s)ds,
    \end{split}
  \end{equation}
  where we used the fact that $\Delta_sG(x) = 2$, $x\in\R$. Noticing that for $s < T$
  \begin{equation}
    \delta X_s^2 - \inv{2}X_sV'(X_s) = \f{X_s^2}{2}[2\delta - gX_s^2 - 1]\le \f{X_s^2}{2}[2\delta - gM^2 - 1] = \f{X_s^2}{2}\left(2\delta - \f{2+\epsilon}{3}\right)
  \end{equation}
  as self-adjoint operators, because $\|X_s\| \le M$ (notice that $-g> 0$).
  Therefore, with the choice $\delta = \f{2+\epsilon}{6}$, we have
  \begin{equation}
    \delta X_s^2 - \inv{2}X_sV'(X_s) \le 0.
  \end{equation}

 Observing that $\|X_s\otimes 1_{\mathcal{A}} + 1_{\mathcal{A}}\otimes X_s\|\le 2\|X_s\|\le 2M$ for $s<T$, we deduce by using free Burkholder-Gundy inequality that
  \begin{equation}
  e^{\delta t}\|X_t^2\| \le \|X_0^2\| + 4\sqrt{2}M\sqrt{\f{e^{2\delta t} - 1}{2\delta}}  + \f{e^{\delta t} - 1}{\delta}.
  \end{equation}

  Consequently, for all $t < T$
  \begin{equation}
    \|X_t\|^2\le m^2 + 4M\sqrt{\f{6}{2+\epsilon}} + \f{6}{2+\epsilon} = m^2 +4\sqrt{6}\sqrt{\f{1-\epsilon}{2+\epsilon}}h + \f{6}{2+\epsilon}.
  \end{equation}

 If the parameters $m, h$ and $\epsilon$ satisfy
 \begin{equation}
   \label{biane-cond-support}
m^2 +4\sqrt{6}\sqrt{\f{1-\epsilon}{2+\epsilon}}h + \f{6}{2+\epsilon} < M = (1-\epsilon)h^2,
 \end{equation}
 then we will have, by the continuity of the norm of the free stochastic integral $(X_t)_{t\ge 0}$, that $T=\infty$ and the support of the distribution of $X_t$ will be bounded.

 Condition (\ref{biane-cond-support}) is satisfied for $m$ and $-g$ small enough. Indeed, we have that
 \begin{equation}
   (1-\epsilon)h^2 - 4\sqrt{6}\sqrt{\f{1-\epsilon}{2+\epsilon}}h - \f{6}{2+\epsilon} > 0
 \end{equation}
 as soon as
 \begin{equation}
   \label{eq:h_star}
   h > \f{2\sqrt{6} + \sqrt{30}}{\sqrt{(1-\epsilon)(2 + \epsilon)}}.
 \end{equation}

 We are interested in finding the smallest $g^*$ such that we can define solutions if the initial support is included in a small neighborhood of the origin. To achieve this, we minimize the right of side of (\ref{eq:h_star}) by taking $\epsilon = 0$.
 Therefore, if 
 \begin{equation}
   - \inv{81 + 36\sqrt{5}} < g < 0
 \end{equation}
 and the initial support satisfies
 \begin{equation}
   m^2 < h^2 - 4\sqrt{3}h - 3= - \inv{3g} - \f{4}{\sqrt{-g}} - 3.
 \end{equation}
 then $T = \infty$ and for all $t\ge 0$
 \begin{equation}
   \|X_t\|^2\le m^2 + \f{4}{\sqrt{-g}} + 3.
 \end{equation}
 This proves (\ref{biane-support}). Finally, we have for all $t\ge 0$
   \begin{equation}
     \inf_{x\in supp(\mu_t)}V''(x) = 1 + 3g\left(m^2 + \f{4}{\sqrt{-g}} + 3\right) = -3g \left(h^2 - m^2 -  \f{4}{\sqrt{-g}} - 3\right) > 0.
   \end{equation}
   This proves (\ref{biane-convexity-set}).
\end{proof}

An immediate consequence is the following lemma.

\begin{lem}
  Let $g\in \left(- \inv{81 + 36\sqrt{5}} , 0\right)$ and $m < m(g)$. The distributions $(\mu_t)_{t\ge 0}$ associated with the process $(X_t)_{t\ge 0}$ have a lower-bounded entropy
  \label{lem:biane}
  \begin{equation}
    \inf_{t\ge 0} \Sigma(\mu_t) > -\infty.
  \end{equation}
  Moreover, the entropy dissipation is still given by $D$
  \begin{equation}
    \f{d}{dt}\Sigma(\mu_t) = - D(\mu_t).
  \end{equation}
\end{lem}
\begin{proof}
  The non-confining potential $V$ (\ref{eq:non_conf_V}) does not satisfy the growth assumption on the whole real line (\ref{growth-V}) of Theorem 3.1 and Proposition 6.1 \cite{biane2001free}. However, the support of the solution being bounded, it can be still satisfied locally. More precisely, there exists a modified potential $\overline{V}$ satisfying
  \begin{equation}
    \begin{cases}
      \overline{V}\in C^2(\R)\\
      \overline{V}(x) = V(x),\textrm{ for } x\in [-M - 1, M + 1]\\
      \overline{V}(x) \ge x^4,\textrm{ for } x\in [-M - 2, M + 2].
    \end{cases}
  \end{equation}
  $\overline{V}$ will therefore satisfy the growth assumption of Theorem 3.1 \cite{biane2001free}. $(\mu_t)_{t\ge 0}$ satisfies the Fokker-Planck equation
  \begin{equation}
    \partial_t\mu_t = \partial_x\left[\mu_t\left(\inv{2}\overline{V}' - H\mu_t\right)\right].
  \end{equation}
  The entropy and dissipation of $\mu_t$, $t\ge 0$ are given respectively by
  \begin{equation}
    \begin{split}
      \Sigma(\mu_t) &= \inv{2}\int_{\R}V(x)\mu_t(dx) - \inv{2}\int_{\R\times\R}\log{|x-y|}\mu_t(dx)\mu_t(dy)\\
      &=\inv{2}\int_{\R}\overline{V}(x)\mu_t(dx) - \inv{2}\int_{\R\times\R}\log{|x-y|}\mu_t(dx)\mu_t(dy)
    \end{split}
  \end{equation}
  and 
  \begin{equation}
    D(\mu_t) = \int_\R \left|\inv{2}\overline{V}'(x)-H\mu(x)\right|^2\mu_t(dx).
  \end{equation}
  The result follows from Theorem 3.1 and Proposition 6.1 \cite{biane2001free} applied to $\overline{V}$.
\end{proof}




We are now ready to prove Theorem \ref{thm:biane}.

\begin{proof}[Proof of Theorem \ref{thm:biane}]
  The measures $(\mu_t)_{t\ge 0}$ have uniformly bounded moments, thanks to the boundedness of the support. Therefore, we have tightness with respect to the Wasserstein distance. Let $\mu_\infty$ be a limit point. According to Lemma \ref{lem:biane}, this limit point is a stationary solution.

  Set $R = \sqrt{m^2 + \f{4}{\sqrt{-g}} + 3}$.
  Let $\rho_0, \rho_1\in\mathcal{M}_2$ with support included in $[-R, R]$. Let $T$ be the optimal transportation map from $\rho_0$ to $\rho_1$. We see that for all $s\in [0, 1]$
  \begin{equation}
    (1-s)x + sT(x)\in [-R, R],\ \forall x\in supp(\rho_0).
  \end{equation}
  Using the fact $\inf\limits_{|x|\le R}V''(x)\ge 2\lambda > 0$, we deduce
  \begin{equation}
    \inv{2}\int V''((1-s)x + sT(x))(T(x) - x)^2\rho_0(dx)\ge \lambda W_2(\rho_0, \rho_1)^2,\ \forall s\in [0, 1].
  \end{equation}
  Following a similar argument to the Remark \ref{rmk:ledoux}, we derive the HWI inequality
  \begin{equation}
    \label{eq:hwi-biane}
    \Sigma(\rho_0)-\Sigma(\rho_1)\le \sqrt{D(\rho_0)}W_2(\rho_0, \rho_1) - \f{\lambda}{2}W_2(\rho_0, \rho_1)^2.
  \end{equation}
  Similarly to Corollary \ref{coro:hwi1}, we deduce exponential stability. To see that $\mu_\infty$ is a local minimizer, take $\rho_0 = \mu_\infty$ and $\rho_1 = \mu$ with support included in $\left[-R, R\right]$. $\mu_\infty$ being a stationary measure, the HWI inequality gives
  \begin{equation}
    \Sigma(\mu_\infty) - \Sigma(\mu) \le -\f{\lambda}{2}W_2(\mu_\infty, \mu)^2\le 0.
  \end{equation}
\end{proof}
\begin{rmk}
  This result can be slightly improved by allowing $g < g^*$. Indeed, we did not exploit the logarithmic energy as we did in Theorem \ref{thm:confining}. More precisely, we do not need $\inf_{x \in supp(\mu_t)}V''(x)>0$ for all $t\ge 0$ to deduce (\ref{eq:hwi-biane}). Indeed, for solutions with a fixed center of mass and with a symmetric initial data, it is enough to ensure that 
  \begin{equation}
    \inf_{x\in supp(\mu_t)}V''(x) + \inv{4\max_{x\in supp(\mu_t)}|x|^2} >0,\ \forall t\ge 0.
  \end{equation}
\end{rmk}

\appendix
\section{Proof of Lemma \ref{lem:power_testf}}

\begin{proof}[Proof of Lemma \ref{lem:power_testf}]
Let $\phi\in \mathcal{C}^\infty(\R, [0, 1])$ such that
\begin{enumerate}
    \item $\phi(0) = 1$ and $\phi(1) = 0$,
    \item $\phi^{(k)}(0) = \phi^{(k)}(1) = 0$ for $k\ge 1$.
\end{enumerate}
Set $|\phi^{(k)}|_\infty = \sup_{x\in[0,1]}|\phi^{(k)}(x)|$.
Define for $K \ge 1$:
\[ \eta_K(x) = \left\{\begin{array}{cl}
    1 &\textrm{ if $|x| \le K$} , \\
    \phi(|x| - K)&\textrm{ if $K\le  |x| \le K + 1,$}\\
    0&\textrm{if $|x|>K+1$}.
\end{array}\right. \]

We check easily that $\eta_K\in \mathcal{C}^\infty(\R, [0, 1])$ and that for all $k\ge 1$,
\[  |\eta_K^{(k)}|_\infty \le C_k |\phi^{(k)}|_\infty\]
for some $C_k > 0$.

Setting $f:x\mapsto |x|^p$ and taking $f_K:x\mapsto \eta_K(x)f(x)$ in (\ref{eq:fokker_planck_bis}), gives
\begin{equation}
\label{eq:lemma_test_ABC}
\begin{split}
\underbrace{\int_{\R} f_K(x)\mu_t(dx) - \int_{\R} f_K(x)\mu_0(dx)}_{\equiv A(K)} &= \inv{2}\int_0^t\underbrace{\iint_{\R\times\R}\frac{f'_K(x)-f'_K(y)}{x-y}\mu_s(dx)\mu_s(dy)}_{\equiv B(s, K)}ds\\
&- \int_0^t\underbrace{\int_{\R} V'(x)f'_K(x)\mu_s(dx)}_{\equiv C(s, K)}ds.\\
\end{split}
\end{equation}
Using that $|f_K|\le |f|$ and the dominated convergence theorem, we see that the left-hand side $A(K)$ converges towards $\int_{\R} |x|^p\mu_t(dx) - \int_{\R} |x|^p\mu_0(dx)$ when $K\rightarrow +\infty$.

Likewise, $C(s, K) = \int_{\R}V'(x)\cdot[|x|^p \eta'_K(x) + p|x|^{p-2}\eta_K(x)x]\mu_s(dx)$ converges towards
\[ C(s, \infty) = \int_{\R}pV'(x)\cdot x|x|^{p-2}\mu_s(dx).\]
Moreover
\[ C(s, K) \le \max(|\phi'|_\infty, |\phi|_\infty)\int_{\R} |V'(x)|(|x|^p + p|x|^{p-1})\mu_s(dx) \in L^1([0, t]).\]

According to the dominated convergence theorem
\begin{equation}
    \label{eq:lemma_test_C}
     \int_0^t C(s, K)ds \xrightarrow{K\rightarrow +\infty} \int_0^t\int_{\R} V'(x)\cdot px|x|^{p-2}\mu_s(dx).
\end{equation}

We treat $B(s, K)$ in a similar way
\begin{equation}
    \label{eq:B(s, K)}
    \begin{split}
    \frac{|(\eta_Kf)'(x) - (\eta_Kf)'(y))\cdot(x - y)|}{|x-y|^2}\le |\phi''|_\infty|f(y)| &+ |\phi|_\infty\frac{|f'(x)-f'(y)|}{|x-y|}\\ 
    &+ |\phi'|_\infty\bigg[|f'(y)| + \frac{|f(x)-f(y)|}{|x-y|}\bigg].\\
    \end{split}
\end{equation}

We check that for $f(x)= |x|^p$ and $\int|x|^p\mu_t(dx)<+\infty$, the right-hand side is integrable. Therefore, $B(s, K)$ converges towards
\[ B(s, \infty) = \iint_{\R\times\R}\frac{f'(x)-f'(y)}{x-y}\mu_s(dx)\mu_s(dy).\]

Using (\ref{eq:B(s, K)}), we find an integrable function in $L^1([0, t])$ dominating $\iint_{\R}B(s, K)\mu_s(dx)$, and thereupon
\begin{equation}
    \label{eq:lemma_test_B}
    \int_0^t B(s, K)ds\xrightarrow{K\rightarrow +\infty} \int_0^t\iint_{\R\times\R}\frac{ f'(x)-f'(y)}{x-y}\mu_s(dx)\mu_s(dy)ds.
\end{equation}
Finally, we conclude from (\ref{eq:lemma_test_ABC}), (\ref{eq:lemma_test_C}) and (\ref{eq:lemma_test_B}), that equation (\ref{eq:fokker_planck_bis}) is satisfied for $f(x) = |x|^p$.

\end{proof}

\section{Proof of Lemma \ref{lem:tail-rho-s}}
\begin{proof}[Proof of Lemma \ref{lem:tail-rho-s}]
  Write
  \begin{equation}
    \int_{|x|\le r}x^2\rho_s(dx) = \int x^2\rho_s(dx) - \int_{|x|>r}x^2\rho_s(dx).
  \end{equation}
  Treat each term independently. First
  \begin{equation}
 \int x^2\rho_s(dx) = \int ((1-s)x + sT(x))^2\rho_0(dx) \ge \inv{2}\min\left( \int x^2\rho_0(dx), \int x^2\rho_1(dx)\right) + 2s(1-s)\int xT(x)\rho_0(dx).
  \end{equation}
  Introduce $q$ the median of $\rho_0$ satisfying $F_{\rho_0}(q)=\inv{2}$. We see immediately by the assumption on $\rho_1$ that $T(q)= 0$. Therefore, $T$ being the gradient of a convex function
  \begin{equation}
    \int xT(x)\rho_0(dx) = \int (x-q)(T(x)-T(q))\rho_0(dx) + q\int T(x)\rho_0(dx)\ge q\int x\rho_1(dx)=0.
  \end{equation}
  Consequently
  \begin{equation}
    \int x^2\rho_s(dx)\ge \inv{2}\min\left( \int x^2\rho_0(dx), \int x^2\rho_1(dx)\right).
  \end{equation}
  For the second term, write
  \begin{equation}
    \int_{|x|>r}x^2\rho_s(dx)\le \inv{r^2}\int x^4\rho_s(dx)\le \f{8}{r^2}\max\left( \int x^4\rho_0(dx), \int x^4\rho_1(dx)\right).
  \end{equation}
\end{proof}

\newpage
\bibliographystyle{plain}  
\bibliography{bib}

\end{document}